\documentclass[12pt,reqno]{amsart}
\usepackage{latexsym,amsmath,amssymb, amsthm, mathscinet}
\usepackage{cases, verbatim}
\usepackage[dvipdfmx]{graphicx}
\usepackage{bmpsize}

\usepackage{amsfonts,amsmath,amsthm, amssymb}
\usepackage{cases}
\usepackage[usenames]{color}
\usepackage{enumerate}
\usepackage{bm}
\usepackage{graphicx}

\theoremstyle{plain}
\newtheorem{thm}{Theorem}[section]

\newtheorem{defn}[thm]{Definition}

\newtheorem{ex}[thm]{Example}
\newtheorem{lem}[thm]{Lemma}
\newtheorem{cor}[thm]{Corollary}
\newtheorem{prop}[thm]{Proposition}
\theoremstyle{remark}

\numberwithin{equation}{section}

\newcommand{\dist}{{\rm dist}\,}

\begin{document}
\title[Static Class-Guided Selection of Elementary Solutions]{Static Class-Guided Selection of Elementary Solutions in Non-Monotone Vanishing Discount Problems}

\author{Panrui Ni}
\address[P. Ni]{
Department 1: Graduate School of Mathematical Sciences, University of Tokyo, 3-8-1 Komaba, Meguro-ku, Tokyo 153-8914, Japan; Department 2: Shanghai Center for Mathematical Sciences, Fudan University, Shanghai 200438, China}
\email{panruini@gmail.com}

\author{Jun Yan}
\address[J. Yan]{
	School of Mathematical Sciences, Fudan University, Shanghai 200433, China}
\email{yanjun@fudan.edu.cn}

\author{Maxime Zavidovique}
\address[M. Zavidovique]{
	Sorbonne Universit\'e, Universit\'e de Paris Cit\'e, CNRS, \ Institut de Math\'ematiques de Jussieu-Paris Rive Gauche, Paris 75005, France}
\email{mzavidovi@imj-prg.fr}


\makeatletter
\@namedef{subjclassname@2020}{\textup{2020} Mathematics Subject Classification}
\makeatother

\date{\today}


\begin{abstract}
We study a generalized vanishing discount problem for Hamilton--Jacobi equations, removing the standard monotonicity assumption, either in a global sense or when integrated against all Mather measures. Specifically, we consider
\[
\lambda a(x)u(x)+H(x,Du(x))-A\lambda=c_0,
\]
with a suitably chosen constant $A>0$. By appropriately changing the signs of the function $a(x)$ on different static classes associated with $H$, we show that the maximal viscosity solution converges uniformly as $\lambda\to 0^+$ and that all elementary solutions of the stationary equation
\[
H(x,Du(x))=c_0
\]
can be selected as limits. This provides the first result for selecting multiple viscosity solutions in vanishing discount problems beyond the usual monotonicity and integral assumptions, as long as $a(x)$ is positive on one static class. Our results highlight the crucial role of static classes in controlling the asymptotic behavior of viscosity solutions. Previously, under usual monotonicity assumptions, only a single solution could be selected (as discussed in \cite{GL}), whereas our approach allows controlled selection of multiple solutions via static class-guided discount coefficients.
\end{abstract}

\date{\today}

\maketitle



\section{Introduction}

In this paper, we investigate whether all elementary solutions of
\begin{equation}\label{E0}\tag{HJ$_0$}
  H(x,Du(x))=c_0\quad \textrm{in}\quad M
\end{equation}
can be obtained as the limit of the maximal viscosity solution of
\begin{equation}\label{E}\tag{E$_\lambda$}
  \lambda a(x)u(x)+H(x,Du(x))-A\lambda=c_0 \quad \textrm{in}\quad M
\end{equation}
as $\lambda\to 0^+$. Here $M$ is a closed Riemannian manifold, $H\in C^2(T^*M)$ is strictly convex and superlinear in the gradient variable, $a(x)$ is continuous, and $c_0$ is the critical value associated with $H$. The constant $A>0$ is chosen sufficiently large so that \eqref{E} admits viscosity solutions. Equivalently, \eqref{E} can be written as
\[G_\lambda(x,Du(x),u(x))=c_0\quad \textrm{in}\quad M, \quad G_\lambda(x,p,u)=\lambda(a(x)u-A)+H(x,p),\]
where $G_\lambda$ converges uniformly to $H(x,p)$ on compact subsets of $T^*M\times\mathbb{R}$ as $\lambda\to 0^+$; see \cite{V1} for related discussion. Throughout the paper, solutions, subsolutions, and supersolutions are meant in the viscosity sense and assumed continuous. Relevant definitions are recalled in Section \ref{pre}.

To obtain the selection result, we introduce a dynamical concept called \emph{static classes}, a natural partition of the Aubry set into disjoint compact invariant subsets following Ma\~n\'e. The Mather set, being contained in the Aubry set, inherits a corresponding partition by static classes. Although the Aubry set may be strictly larger than the Mather set, each static class contains at least one point in the Mather set, or equivalently, supports at least one ergodic minimizing measure. As a consequence, the number of static classes in the Aubry set coincides with that in the Mather set. As shown in \cite{Ber}, there is a one-to-one correspondence between static classes and elementary solutions of \eqref{E0}. In this paper, we assume finitely many static classes, more than one, which holds generically for $H$. We prove that by setting $a(x)>0$ on one static class $M_{i_0}$, and $a(x)<0$ on all other static classes, the maximal viscosity solution of \eqref{E} converges to the elementary solution associated with $M_{i_0}$ as $\lambda\to 0^+$. Consequently, by varying the sign pattern of $a(x)$ on the static classes, all elementary solutions of \eqref{E0} can be selected.

Recently, \cite{QC} studied the selection problem for \eqref{E0} through a two-step perturbation involving a vanishing discount term and a small potential perturbation. This approach was the first to select different solutions of \eqref{E0}, where the potential term plays a key role in removing unwanted minimizing measures. In the present work, by carefully choosing the signs of the discount coefficient $a(x)$ across different static classes, one can select different elementary solutions relying solely on the vanishing discount term. This result emphasizes a purely dynamical mechanism driving the selection process and offers an intrinsic perspective on the problem.

\subsection*{History of the problem}

The existence of solutions to the stationary equation
\begin{equation}\label{hjc}
H(x,Du(x))=c \quad \textrm{in}\quad \mathbb T^d
\end{equation}
is known as the \emph{cell problem}, which is a central issue in the theory of Hamilton--Jacobi equations. 
This problem was solved via the so-called ergodic (or discounted) approximation in \cite{hom}. 
Let $\lambda>0$ and let $u_\lambda$ be the unique solution of
\[
\lambda u(x)+H(x,Du(x))=0 \quad \textrm{in}\quad \mathbb T^d.
\]
It was shown in \cite{hom} that there exists a sequence $\lambda_k\to0^+$ such that $-\lambda_k u_{\lambda_k}$ converges uniformly to a constant $c_0$, and $u_{\lambda_k}-\min_{\mathbb T^d}u_{\lambda_k}$ converges uniformly to a solution of \eqref{hjc} with $c=c_0$. 
Moreover, $c_0$ is the unique constant for which \eqref{hjc} admits solutions and is called the \emph{critical value}.

At that time, it was unclear whether different sequences $\lambda\to0^+$ would lead to the same limit. 
This question was first studied under restrictive assumptions in \cite{G,IM}. 
A complete positive answer was later obtained in \cite{Da1}, where the uniform convergence of the unique solution $u_\lambda$ of
\begin{equation}\label{lde}\tag{${\textrm{HJ}}_\lambda$}
\lambda u(x)+H(x,Du(x))=c_0 \quad \textrm{in}\quad M
\end{equation}
as $\lambda\to0^+$ was established. 
Here $M$ is a closed connected manifold and $H$ is continuous, coercive and convex in the momentum variable. 
This type of problem is referred to as the \emph{vanishing discount problem}. 
The convexity of $H$ plays a crucial role, and convergence may fail without this assumption; see \cite{Z2}.

The asymptotic convergence result has since been extended to a variety of settings, including second-order equations \cite{IMT,IMT2,MT,Zh}, discrete models \cite{Da1,n1,Su,Zbook}, mean field games \cite{CP,IMWX}, weakly coupled Hamilton--Jacobi systems \cite{Da6,Da5,Ish4,Ish5}, and non-compact manifolds \cite{Da7,Ish6}. 
For nonlinear generalizations, known as the vanishing contact structure problem, we refer to \cite{V3,V1,GMT,V2}. 
Other extensions of the vanishing discount problem have been investigated in \cite{TZ,WZ} and the references therein.

As a degenerate but still monotone case, \cite{Z} studied the convergence of solutions to
\begin{equation}\label{gendis}\tag{$\overline{\textrm{HJ}}_\lambda$}
\lambda a(x)u(x)+H(x,Du(x))=c_0 \quad \textrm{in}\quad M,
\end{equation}
where $a(x)\geqslant 0$ on $M$ and $a(x)>0$ on the projected Aubry set of $H$. 
This equation is closely related to the present work and is also connected to optimization problems in economics; see \cite{factor}. 
Later, the degenerate vanishing contact structure problem was investigated in \cite{CFZZ}, where it was shown that Mather measures play a central role in the convergence analysis.

It is worth pointing out that all the works mentioned above rely on a non-decreasing (monotonicity) assumption with respect to the unknown function $u$. 
Once this assumption is violated, many fundamental tools break down: solutions may fail to exist, comparison principles no longer hold, and uniqueness is generally lost. 
In \cite{Da3,V4}, the convergence of minimal solutions of \eqref{lde} as $\lambda\to0^-$ was studied under restrictive hypotheses. 
For genuinely non-monotone vanishing discount problems, the first example exhibiting both convergent and divergent families of solutions was provided in \cite{n3}, revealing phenomena absent from the monotone theory. 
Subsequently, the generalized vanishing discount problem without the non-decreasing assumption was further developed in \cite{DNYZ}.

\subsection*{Motivation of this paper}

Previous works mainly focus on the convergence of discounted solutions 
\footnote{A \emph{discounted solution} refers to a viscosity solution of the generalized discounted equation \eqref{gendis}, while a \emph{critical solution} refers to a viscosity solution of \eqref{E0}.}. 
In contrast, the present paper addresses a more refined problem: how to select \emph{different elementary critical solutions} via limits of discounted solutions. 
As observed in \cite{DNYZ}, there exists a certain critical solution that is the only possible limit for all bounded discounted solutions, so previous works could only select this particular solution. 

In \cite{GL}, a selection principle based on the Freidlin--Wentzell large deviation principle was proposed, yielding a critical solution different from the one obtained via the vanishing discount process. 
For instance, as in \cite[Section~5.3]{GL}, we consider
\begin{equation}\label{U'}
  \lambda u(x)+u'(x)(u'(x)-U'(x))=0 \quad \textrm{in}\quad \mathbb S^1\simeq [0,1),
\end{equation}
where $U:\mathbb S^1\to\mathbb R$ is smooth. The unique discounted solution is $u_\lambda\equiv 0$, so the selected critical solution $u_0$ is trivial. 
In this paper, we show that the \emph{generalized vanishing discount process} can be used to select \emph{all elementary critical solutions}, including nontrivial ones for \eqref{e12} (see Example \ref{ex1}). 
This represents a natural development of the vanishing discount problem, highlighting for the first time how the dynamical structure of static classes within the Mather sets governs the asymptotic behavior of discounted solutions.

Recall that static classes are in one-to-one correspondence with elementary solutions of \eqref{E0}, which play a fundamental role in the description of critical solutions. This makes it natural to consider a function $a(x)$ in \eqref{gendis} whose sign varies across different static classes: by changing the signs of $a(x)$, one can select different elementary solutions through the generalized vanishing discount process.

According to \cite{DNYZ,n3}, one cannot in general expect all discounted solutions of \eqref{E} to converge as $\lambda$ vanishes. 
This motivates us to consider the convergence of the \emph{maximal discounted solution} of \eqref{E}. 
A second key novelty of this paper is that we establish convergence even when the discount coefficient $a(x)$ is allowed to take negative values outside a chosen static class, so that its integral may be negative for some Mather measures. 
This introduces essential technical challenges, as previous works \cite{CFZZ,DNYZ} rely on the assumption that $a(x)$ is strictly positive when integrated against all Mather measures. Removing this monotonicity condition imposed on all Mather measures gives rise to genuinely new selection phenomena.

Finally, a key ingredient in our analysis is a large-time behavior result (Lemma \ref{stab}), building on earlier work \cite{n2}. To our knowledge, this reveals for the first time a connection between the vanishing discount problem and the large-time behavior.

\subsection*{Statement of the main result}

Let $M$ be a closed connected smooth manifold. Denote by $TM$ and $T^*M$ the tangent and cotangent bundles of $M$, with points $(x,v)\in TM$ and $(x,p)\in T^*M$. Let $\|\cdot\|_x$ be the norm induced by the Riemannian metric on $T_xM$ and $T^*_xM$, and let $\pi:TM\to M$ be the canonical projection. Assume that $H:T^*M\to\mathbb R$ is $C^2$ and satisfies the Tonelli conditions:
\begin{itemize}
\item [(H1)] \textbf{Strict convexity:} $\partial^2 H/\partial p^2(x,p)$ is positive definite for all $(x,p)\in T^*M$.
\item [(H2)] \textbf{Superlinearity:} $\lim_{\|p\|_x\to+\infty} H(x,p)/\|p\|_x = +\infty$ uniformly in $x$.
\end{itemize}

The associated Lagrangian $L:TM\to\mathbb R$ is
\[
L(x,v)=\sup_{p\in T^*_xM} (p\cdot v - H(x,p)),
\]
which is $C^2$ and satisfies
\begin{itemize}
\item [(L1)] \textbf{Strict convexity:} $\partial^2 L/\partial v^2(x,v)$ is positive definite for all $(x,v)\in TM$.
\item [(L2)] \textbf{Superlinearity:} $\lim_{\|v\|_x\to+\infty} L(x,v)/\|v\|_x = +\infty$ uniformly in $x$.
\end{itemize}
Since ``static classes'' is a dynamical concept, we assume the Hamiltonian to be smooth. 
The authors believe that the results in the present paper can be generalized to the case where the Hamiltonian is merely continuous. 
This condition is standard from the PDE point of view; for related tools, one can refer to \cite{gen,DZ10,FS}. 
However, treating this more general case would introduce technical complications that could obscure the main ideas of this work. 
We therefore leave this generalization for future study. 
Moreover, understanding the concept of ``static classes'' and the selection result from a purely PDE perspective remains both unclear and interesting.

Let $h^\infty : M \times M \to \mathbb{R}$ be the Peierls barrier defined in \eqref{barr}. 
The following pseudo-distance was first introduced in \cite{Ma3}; see also \cite{global} for a detailed discussion:
\[
d_H(x,y) := h^\infty(x,y) + h^\infty(y,x).
\]
Let $\mathcal A$ be the projected Aubry set defined in \eqref{pA}. 
We define an equivalence relation on $\mathcal A$ by declaring $x \sim y$ if $d_H(x,y)=0$. 
The equivalence classes are called the \emph{static classes}. 
Let $\mathcal M$ be the projected Mather set defined in Proposition \ref{Mather}. 
Since $\mathcal M \subseteq \mathcal A$, it inherits the partition into static classes. 
Assume:

\begin{itemize}
\item [($\diamond$)] The number of static classes is finite and greater than one.
\end{itemize}
Since each static class supports at least one ergodic minimizing measure and the Mather set is the union of the supports of all minimizing measures, the number of static classes in $\mathcal M$ coincides with that in $\mathcal A$. 
The assumption $(\diamond)$ is common in the study of dynamics; see \cite{Ber,CP2}. 
Generically, according to \cite{BC}, one has finitely many static classes. 
The requirement that the number of static classes be greater than one is imposed to ensure that the selection result is nontrivial; otherwise, there exists only one elementary solution of \eqref{E0}. 

A special case of assumption $(\diamond)$ is provided by symmetric Hamiltonians satisfying
$H(x,p)=H(x,-p)$. 
In this case, it is known that
\[
\mathcal A=\mathcal M=\Big\{x\in M:\ H(x,0)=\max_{y\in M} H(y,0)\Big\},
\]
see \cite{Da3,F}. 
Consequently, if $H(x,0)$ admits only finitely many maximizers, then the Aubry set (and hence the Mather set) consists of finitely many points, each forming a distinct static class. 
Therefore, assumption $(\diamond)$ is satisfied in this setting.

The case where there are infinitely many static classes is still unclear, and we leave this problem for future study. 
In fact, this issue is closely related to the total disconnectedness
\footnote{That is, each connected component of the set consists of a single point.}
of the quotient Aubry set induced by the equivalence relation defined above; see, for instance, \cite{FFR,Ma1,So}. 
In general, however, this property does not hold; cf. \cite{BIK,Ma4}.

Label the static classes in $\mathcal M$ as $M_1,\dots,M_k$, with $1<k<\infty$, so that
\[
\mathcal M=\bigcup_{i=1}^k M_i, \quad M_i\cap M_j=\emptyset\ (i\neq j).
\]
The elementary solution of \eqref{E0} associated with $M_i$ is, up to a constant, $h^\infty(x_i,x)$ for any $x_i\in M_i$; see Corollary \ref{h-h=c} below.

\medskip

The main result of this paper is as follows.

\begin{thm}\label{thm1}
Assume \emph{(H1)}, \emph{(H2)}, and \emph{($\diamond$)}.
Let $v_0$ be a viscosity solution of \eqref{E0} and take
$A > \|a\|_\infty \|v_0\|_\infty$ in \eqref{E}. Then \eqref{E} admits a maximal solution $u_\lambda$. Fix $i_0 \in \{1,\dots,k\}$ and assume
\begin{itemize}
\item[(a)] $a(x)>0$ on $M_{i_0}$ and $a(x)<0$ on $\mathcal M \setminus M_{i_0}$.
\end{itemize}
Then the maximal solution $u_\lambda$ converges uniformly as $\lambda \to 0^+$ to
\[h^\infty(x_0,x) + C,\]
where $x_0 \in M_{i_0}$ is arbitrary and the constant $C$ is given by
\[C=\inf_{\mu}\frac{\int_M a(y)\, h^\infty(y,x_0)\, d\mu(y) + A}{\int_M a(y)\, d\mu(y)},\]
with the infimum over all projected Mather measures supported in $M_{i_0}$.
\end{thm}

Now we give two examples. The first one was discussed in \cite[Section 5.3]{GL}.
\begin{ex}\label{ex1}
Consider
\begin{equation}\label{e11}
  \lambda a(x)u(x)+u'(x)\bigl(u'(x)-U'(x)\bigr)=A\lambda
  \quad \textrm{in}\quad \mathbb S^1\simeq [0,1),
\end{equation}
where $U:\mathbb S^1\to\mathbb R$ is smooth and has exactly two critical points $0$ (minimum) and $X$ (maximum). Then there are two classical solutions (up to the addition of a constant)
\[
u_1(x)\equiv 0,
\qquad
u_2(x)=U(x)
\]
of
\begin{equation}\label{e12}
  u'(x)\bigl(u'(x)-U'(x)\bigr)=0.
\end{equation}
Other viscosity solutions of \eqref{e12} are given by
\[
u(x):=\min\{C_1,U(x)+C_2\},
\]
where $C_1$ and $C_2$ are constants.
These solutions may be non-smooth at points in $\mathbb S^1\setminus\{0,X\}$.
Since all critical solutions are differentiable on the Aubry set, we have
\[
\mathcal A=\mathcal M=\{0,X\}.
\]
As pointed out in Proposition \ref{h>w}, for $y\in\mathcal A$, $h^\infty(y,\cdot)$ is the maximal subsolution $w$ of \eqref{e12}
satisfying $w(y)=0$.
The elementary solutions of \eqref{e12} are therefore
\[
h^\infty(0,x)=U(x)-U(0),
\qquad
h^\infty(X,x)\equiv 0.
\]
Since $u_1(x)\equiv0$ is a solution of \eqref{e12}, we take $A>0$ in \eqref{e11}.
Then, applying Theorem \ref{thm1}, if $a(0)>0$ and $a(X)<0$ (resp. $a(0)<0$ and $a(X)>0$), the maximal solution of
\eqref{e11} converges uniformly, up to an additive constant, to the elementary solution
$u_2(x)=U(x)$ (resp. $u_1(x)\equiv0$). The critical solution obtained by previous vanishing discount approaches coincides with $u_1(x)\equiv 0$, as noted in \cite{GL}.

This discussion extends to the case where $U:\mathbb S^1\to\mathbb R$ has finitely many minimal and maximal points $\{x_i\}_{i=1}^k$. According to \cite[Lemma 4.1]{GL}, the projected Aubry set $\mathcal A=\{x_i\}_{i\in\{1,\dots,k\}}$. By \cite[Proposition 3-11.4]{global}, every static class in the projected Aubry set is connected. Thus, each $x_i$ forms a distinct, connected static class. Choosing $a(x_i)>0$ and $a(x_j)<0$ for all $j\neq i$ selects the elementary solution $h^\infty(x_i,x)$ associated with $x_i$.
\end{ex}

\begin{ex}
Consider
\begin{equation}\label{e21}
  \lambda a(x)u(x)+(u'(x))^2-U(x)-A\lambda=0
  \quad \textrm{in}\quad \mathbb S^1\simeq [0,1),
\end{equation}
where $U:\mathbb S^1\to\mathbb R$ satisfies
\[
U\geqslant 0,\qquad \min_{\mathbb S^1}U=0,
\qquad U(0)=U(X)=0\qquad U(x)>0\text{ for }x\neq 0,X,
\]
Then $\mathcal A=\mathcal M=\{0,X\}$. Let $\tilde U$ be the periodic extension of $U$ to $\mathbb R$ and define
\[
p_\pm(x):=\pm\sqrt{\tilde U(x)}.
\]
There exists $s_1\in[0,1)$ such that
\[
\int_0^{s_1}p_+(x)\, dx+\int_{s_1}^1 p_-(x)\, dx=0,
\]
and $s_2\in[X,X+1)$ such that
\[
\int_X^{s_2}p_+(x)\, dx+\int_{s_2}^{X+1}p_-(x)\, dx=0.
\]
Define
\[
u_1(x)=
\begin{cases}
\displaystyle \int_0^x p_+(y)\, dy, & x\in[0,s_1],\\[1ex]
\displaystyle -\int_x^1 p_-(y)\, dy, & x\in(s_1,1),
\end{cases}
\]
and
\[
u_2(x)=
\begin{cases}
\displaystyle \int_X^x p_+(y)\, dy, & x\in[X,s_2],\\[1ex]
\displaystyle -\int_x^{X+1} p_-(y)\, dy, & x\in(s_2,X+1).
\end{cases}
\]
After periodic extension and projection onto $\mathbb S^1$, $u_1$ and $u_2$ are viscosity
solutions of
\begin{equation}\label{e22}
  (u'(x))^2-U(x)=0
  \quad \textrm{in}\quad \mathbb S^1.
\end{equation}
Other viscosity solutions of \eqref{e22} are given by
\[
u(x):=\min\{u_1(x)+C_1,u_2(x)+C_2\}.
\]
The elementary solutions of \eqref{e22} are
\[
h^\infty(0,x)=u_1(x),
\qquad
h^\infty(X,x)=u_2(x).
\]
By Theorem \ref{thm1}, if we choose $A>\|a\|_\infty\|u_1\|_\infty$ and $a(0)>0$, $a(X)<0$ (resp. $a(0)<0$, $a(X)>0$), the maximal solution of \eqref{e21} converges uniformly, up to an additive constant, to $u_1$ (resp. $u_2$). 
In contrast, previous works on the vanishing discount problem select the critical solution as the supremum of critical subsolutions $w(x)$ satisfying $w(0)\leqslant 0$ and $w(X)\leqslant 0$, which coincides with $u_0(x)=\min\{u_1(x),u_2(x)\}$; see \cite[Section 4.1]{Zbook}.

This example can be generalized to the case where $x$ belongs to a $d$-dimensional closed connected manifold $M$, and the potential $U:M\to\mathbb R$ has finitely many minimal points $\{x_i\}_{i=1}^k$. 
By \cite{Da3,F}, the projected Aubry set is $\mathcal A=\{x_i\}_{i=1}^k$. Since each static class in the projected Aubry set is connected, these points belong to distinct static classes. 
For each index $i\in\{1,\dots,k\}$, taking $a(x_i)>0$ and $a(x_j)<0$ for all $j\neq i$ allows us to select the elementary solution $h^\infty(x_i,x)$ associated with $x_i$.
\end{ex}

\section{Preliminaries}\label{pre}

\subsection*{Viscosity solutions and weak KAM solutions}

In this subsection, we collect several properties of viscosity and weak KAM solutions. We refer the reader to \cite{Barles,guide,Fat12} for further details.

\begin{defn}
Let $G:T^*M\times\mathbb R\to\mathbb R$ be a continuous function and $c\in\mathbb R$. A function $u\in C(M)$ is called a viscosity subsolution (resp. supersolution) of
\begin{equation}\label{hjj}
  G(x,Du(x),u(x))=c\quad \textrm{in}\quad M
\end{equation}
if for each $\phi\in C^1(M)$, when $u-\phi$ attains its local maximum (resp. minimum) at $x$, then
\begin{equation*}
  G(x,D\phi(x),u(x))\leqslant c,\quad (\textrm{resp}.\ \geqslant c).
\end{equation*}
A continuous function $u$ is called a viscosity solution of \eqref{hjj} if it is both a viscosity subsolution and a viscosity supersolution. If there is a constant $\bar c<c$ such that \[G(x,Du(x),u(x))\leqslant \bar c\quad \textrm{in}\quad M\] holds in the viscosity sense, we call $u$ a strict subsolution of \eqref{hjj}.
\end{defn}

\begin{prop}\label{stability}
Let $(G_n)_n$ and $(u_n)_n$ be two sequences of functions in $C(T^*M\times\mathbb R)$ and $C(M)$ respectively. For each $n\in\mathbb N$, $u_n$ is a solution (resp. subsolution, supersolution) of \eqref{hjj} with $G=G_n$. If $u_n\to u$ uniformly and $G_n\to G$ locally uniformly as $n\to+\infty$, then $u$ is a solution (resp. subsolution, supersolution) of \eqref{hjj}.
\end{prop}

\begin{prop}\label{supsub}
If $u$ is the pointwise supremum of a family of subsolutions of \eqref{hjj}, then $u$ is a subsolution of \eqref{hjj}.
\end{prop}

Now we denote by $(x,p,u)$ a point in $T^*M\times\mathbb R$. Assume that $G(x,p,u)$ is continuous, and is convex in $p$ for each $(x,u)\in M\times\mathbb R$.
\begin{prop}\label{ue}
Let $w:M\to\mathbb R$ be a Lipschitz continuous function verifying \[G(x,Dw(x),w(x))\leqslant c\] for almost every $x\in M$. Then for every $\varepsilon>0$, there is $w_\varepsilon\in C^\infty(M)$ such that \[\|w-w_\varepsilon\|_\infty\leqslant \varepsilon,\quad G(x,Dw_\varepsilon(x),w_\varepsilon(x))\leqslant c+\varepsilon\quad \forall x\in M.\]
\end{prop}

Now we further assume that there is $\Theta>0$ such that
\[|G(x,p,u)-G(x,p,v)|\leqslant \Theta|u-v|,\quad \forall (x,p)\in T^*M,\ \forall u,v\in\mathbb R,\]
and \[\lim_{\|p\|_x\to+\infty}\inf_{x\in M}G(x,p,0)=+\infty,\]
then we can define the associated Lagrangian
\[L_G(x,v,u):=\sup_{p\in T^*_xM}(p\cdot v-G(x,p,u)).\]
\begin{defn}\label{bws}
Let $\lambda\in\mathbb R$. A function $u\in C(M)$ is called a backward (resp. forward) weak KAM solution of \eqref{hjj} if
\begin{itemize}
\item [(1)] For each absolutely continuous curve $\gamma:[t',t]\rightarrow M$, we have
\begin{equation*}
  u(\gamma(t))-u(\gamma(t'))\leqslant \int_{t'}^{t}\bigg[L_G(\gamma(s),\dot \gamma(s),u(\gamma(s)))+c\bigg]\, ds.
\end{equation*}
The above condition is denoted by $u\prec L_G+c$.

\item [(2)] For each $x\in M$, there exists an absolutely continuous curve $\gamma_-:(-\infty,0]\rightarrow M$ (resp. $\gamma_+:[0,+\infty)\to M$) with $\gamma_\pm(0)=x$ such that
\begin{align*}
  &u(x)-u(\gamma_-(t))=\int_t^0\bigg[L_G(\gamma_-(s),\dot \gamma_-(s),u(\gamma_-(s)))+c\bigg]\, ds,\quad \forall t<0.
  \\ &\textrm{\Big(resp.}\ u(\gamma_+(t))-u(x)=\int_0^t \bigg[L_G(\gamma_+(s),\dot \gamma_+(s),u(\gamma_+(s)))+c\bigg]\, ds,\quad \forall t>0\textrm{\Big).}
\end{align*}
The curves satisfying the above equality are called $(u,L_G,c)$-calibrated curves.
\end{itemize}
\end{defn}

According to \cite[Appendix D]{NWY} and \cite[Appendix A]{n2}, we have
\begin{prop}\label{soleq}
The following are equivalent:
\begin{itemize}
\item[(i)] $u\prec L_G+c$.
\item[(ii)] $u(x)$ a viscosity subsolution of \eqref{hjj}.
\item[(iii)] $u(x)$ is Lipschitz continuous and $G(x,Du(x),u(x))\leqslant c$ holds almost everywhere.
\end{itemize}
The following are equivalent:
\begin{itemize}
\item[(i)] $u(x)$ a viscosity solution of \eqref{hjj}.
\item[(ii)] $u(x)$ is a backward weak KAM solution of \eqref{hjj}.
\end{itemize}
\end{prop}

\subsection*{Aubry-Mather theory}

In what follows, we always assume that $H:T^*M\to\mathbb R$ is of class $C^2$ and satisfies assumptions (H1) and (H2). We recall that $c_0$ denotes the critical value of $H$, and that $L:TM\to\mathbb R$ is the associated Lagrangian. We now collect several results from weak KAM theory; see \cite{Fathi08}. For extensions to continuous Hamiltonians, we refer to \cite{gen,DZ10,FS}.

For $t>0$, let $h_t:M\times M\to\mathbb R$ be the minimal action function, which is defined as
\[h_t(x,t)=\inf_{\gamma}\int_0^t [L(\gamma(s),\dot\gamma(s))+c_0]\, ds,\]
where the infimum is taken over all absolutely continuous curves $\gamma:[0,t]\to M$ satisfying $\gamma(0)=x$ and $\gamma(t)=y$. By Tonelli's theorem, the infimum is attained; see \cite{One}. The proof relies on the following lower semicontinuity result; see \cite[Theorem 3.5]{One}.
\begin{lem}\label{TM}
Let $J$ be a bounded interval of $\mathbb R$. Assume that $F(t,x,v)$ and $\frac{\partial F}{\partial v}(t,x,v)$ are continuous, $F(t,x,v)$ is convex in $v$, and bounded from below. Then the integral functional
\begin{equation*}
  \mathcal F(\gamma)=\int_J F(s,\gamma(s),\dot \gamma(s))\, ds
\end{equation*}
is sequentially weakly lower semicontinuous in $W^{1,1}(J,M)$, that is, if there is a sequence $(\gamma_n)_n$ weakly converges to $\gamma$ in $W^{1,1}(J,M)$, then
\[\mathcal F(\gamma)\leqslant\liminf_{n\to+\infty} \mathcal F(\gamma_n).\]
Equivalently we can say that the above inequality holds if $(\gamma_n)_n$ uniformly converges to $\gamma$ and the $L^1$-norms of $(\dot{\gamma}_n)_n$ are equi-bounded.
\end{lem}

\begin{prop}\label{htlip}
There is $\kappa>0$ independent of $t$ such that $(x,y)\mapsto h_t(x,y)$ is $\kappa$-Lipschitz continuous for all $t>1$.
\end{prop}
The Peierls barrier $h^\infty:M\times M\to\mathbb R$ is defined as
\begin{equation}\label{barr}
  h^\infty(x,y)=\liminf_{t\to+\infty}h_t(x,y).
\end{equation}
The projected Aubry set is defined as
\begin{equation}\label{pA}
  \mathcal A:=\{x\in M:\ h^\infty(x,x)=0\}.
\end{equation}

\begin{prop}\label{h>w}\cite[Proposition 3.6]{DZ10}.
The following properties hold
\begin{itemize}
\item[(i)] The Peierls barrier is finte valued and Lipschitz continuous.
\item[(ii)] If $w$ is a susolution of \eqref{E0}, then
\[h^\infty(x,y)\geqslant w(y)-w(x).\]
\item[(iii)] For each $x,y,z\in M$, the following triangle inequality holds
\[h^\infty(x,y)\leqslant h^\infty(x,z)+h^\infty(z,y).\]
\item[(iv)] The function $h^\infty(y,\cdot)$ gives a solution of \eqref{E0} for each $y\in M$. Similarly, $-h^\infty(\cdot,y)$ is a forward weak KAM solution and a viscosity subsolution of \eqref{E0} for each $y\in M$.
\end{itemize}
\end{prop}
Let $y\in\mathcal A$. The elementary solution associated with $y$ is given by $h^\infty(y,\cdot)$; see \cite[Section 4.2]{Ber}. By Proposition \ref{h>w} (ii), (iv), $h^\infty(y,\cdot)$ is the maximal subsolution $w$ of \eqref{E0} such that $w(y)=0$.

\begin{defn}
We say that a Borel probability measure $\tilde{\mu}$ on $TM$ is closed if
\begin{itemize}
\item [(1)] $\int_{TM}\|v\|_x\, d\tilde{\mu}(x,v)<+\infty$;
\item [(2)] for all function $f\in C^1(M)$, we have $\int_{TM}Df(v)\, d\tilde{\mu}(x,v)=0$.
\end{itemize}
\end{defn}

\begin{prop}\label{Mather}\cite[Theorem 5.7]{Da1}.
The following holds
\[\min_{\tilde{\mu}}\int_{TM}L(x,v)\, d\tilde{\mu}=-c_0,\]
where $\tilde{\mu}$ is taken among all closed measures on $TM$. Measures realizing the minimum are called Mather measures. We denote by $\widetilde{\mathfrak M}$ the set of all Mather measures. The Mather set is defined as
\[\widetilde{\mathcal M}=\overline{\bigcup_{\tilde{\mu}\in\widetilde{\mathfrak M}}{\rm supp}(\tilde{\mu})}.\]
The projected Mather set is $\mathcal M=\pi(\widetilde{\mathcal M})$.
\end{prop}

\begin{prop}\cite[Proposition 3.13]{Z}.
$\mathcal M\subseteq \mathcal A$.
\end{prop}

\begin{prop}\label{uniqueM}
If two solutions of \eqref{E0} coincide on $\mathcal M$, they coincide on $M$.
\end{prop}

Let $\tilde{\mu}$ be a Mather measure. The associated projected Mather measure $\mu$ is defined by
\[\int_M f(x)\, d\mu(x)=\int_{TM}f(\pi(x,v))\, d\tilde{\mu}(x,v),\quad \forall f\in C(M).\]
Throughout this paper, we will denote by $\tilde\mu$ a probability measure on $TM$, and denote by $\mu$ a probability measure defined on $M$.

\subsection*{Results for contact H-J equations}

Finally, we collect several results from \cite{n2} that will be used in the proof of Theorem \ref{thm1}. Assume that $a\in C(M)$ and that there exist two points $x_1,x_2\in M$ such that $a(x_1)>0$ and $a(x_2)<0$. Let $\varphi\in C(M)$ and $c\in\mathbb R$. We define the solution semigroup $T_t^-:C(M)\to C(M)$ by
\begin{equation}\label{T-}
  T^-_t\varphi(x)=\inf_{\gamma(t)=x} \left\{\varphi(\gamma(0))+\int_0^t\bigg[L(\gamma(\tau),\dot{\gamma}(\tau))+c- a(\gamma(\tau))T^-_\tau\varphi(\gamma(\tau))\bigg]\, d\tau\right\},
\end{equation}
where the infimum is taken among absolutely continuous curves $\gamma:[0,t]\rightarrow M$ with $\gamma(t)=x$. Define the corresponding forward semigroup as
\begin{equation}\label{T+}
  T^+_t\varphi(x)=\sup_{\gamma(0)=x} \left\{\varphi(\gamma(t))-\int_0^t\bigg[L(\gamma(\tau),\dot{\gamma}(\tau))+c- a(\gamma(\tau))T^+_{t-\tau}\varphi(\gamma(\tau))\bigg]\, d\tau\right\}.
\end{equation}
\begin{lem}\label{n11} \cite[Lemma 5.4]{n2}.
If there is a strict subsolution $u_0$ of
\begin{equation}\label{hjj0}
  a(x)u(x)+H(x,Du(x))=c\quad \textrm{in}\quad M,
\end{equation}
then the following limit exists
\[u_-=\lim_{t\to+\infty}T^-_t u_0,\]
and $u_-$ is the maximal solution of \eqref{hjj0}. Also, the following limit exists
\[v_+=\lim_{t\to+\infty}T^+_tu_0,\]
and $v_+$ is the minimal forward weak KAM solution of \eqref{hjj0}. Moreover, $v_+<u_-$.
\end{lem}

\begin{lem}\label{n12}\cite[Theorem 3]{n2}.
Assume that there exists a strict subsolution $u_0$ of \eqref{hjj0}. Let $u_-$ and $v_+$ denote the maximal viscosity solution and the minimal forward weak KAM solution of \eqref{hjj0}, respectively. If $\varphi\in C(M)$ satisfies $\varphi>v_+$, then $T_t^-\varphi$ converges uniformly to $u_-$ as $t\to+\infty$.
\end{lem}

\section{Proof of Theorem \ref{thm1}}

We first recall the setting considered in \cite{DNYZ}, where $a(x)$ was assumed to satisfy
\begin{equation}\label{a>0}
  \int_{TM} a(x)\, d\tilde{\mu} > 0
\end{equation}
for all Mather measures $\tilde{\mu}$ associated with $H$. In the present paper, however, we allow $a(x)$ to change sign on different static classes in the Mather set, which introduces substantial new difficulties.

The first difficulty concerns the existence of solutions to \eqref{E}. For this reason, we introduce the additional term $-A\lambda$ in \eqref{E}. When condition \eqref{a>0} holds, the existence of solutions to \eqref{E} with $A=0$ was established in \cite{DNYZ}.

To prove the convergence of the maximal solution of \eqref{E}, we first rely on the large-time behavior described in Lemma \ref{n12}. This implies that any Mather measure associated with the maximal solution of \eqref{E} satisfies
\[\int_{TM} a(x)\, d\tilde{\mu} \geqslant 0,\]
as shown in Lemma \ref{mu*>0}. However, this information alone is not sufficient to conclude convergence. We therefore employ a dynamical argument in which the structure of static classes plays a crucial role. The key step is Lemma \ref{xi1}, from which Lemma \ref{e1e2} follows. Based on Lemma \ref{e1e2}, we adapt the method introduced in \cite{DNYZ} to prove the convergence of the maximal solution of \eqref{E} on $M_{i_0}$. Using the property established in Lemma \ref{eM1}, we then obtain convergence on the entire Mather set, and hence on the whole space $M$ by Proposition \ref{uniqueM}. Finally, in Lemma \ref{u0h}, we provide a representation of the limit of the maximal solution of \eqref{E} in terms of the Peierls barrier, showing that the selected critical solution coincides with the elementary solution associated with the static class $M_{i_0}$.

\medskip

We first show that \eqref{E} has the maximal solution.
\begin{lem}\label{v0sub}
The solution $v_0$ of \eqref{E0} is a strict viscosity subsolution of \eqref{E}.
\end{lem}
\begin{proof}
Note that $A > \|a\|_\infty \|v_0\|_\infty$. A direct calculation then gives
\[\lambda a(x)v_0(x)+H(x,Dv_0)-A\lambda\leqslant c_0-(A-\|a\|_\infty\|v_0\|_\infty)\lambda<c_0\quad \text{a.e. in } M.\]
It follows from Proposition \ref{soleq} that $v_0$ is a strict viscosity subsolution of \eqref{E}.
\end{proof}

Let $T^\lambda_t$ and $T^{\lambda,+}_t$ be the semigroups defined in \eqref{T-} and \eqref{T+}, associated with
\[\bar L(x,v,u):=L(x,v)+c_0-\lambda a(x)u+A\lambda,\]
respectively. By Lemma \ref{n11}, the limit
\[u_\lambda=\lim_{t\to+\infty}T^\lambda_t v_0,\]
exists, and $u_\lambda$ is the maximal solution of \eqref{E}. Also, the limit
\[v^+_\lambda=\lim_{t\to+\infty}T^{\lambda,+}_tv_0,\]
exists, and $v^+_\lambda$ is the minimal forward weak KAM solution of \eqref{E}, with $v^+_\lambda < u_\lambda$. Moreover, Lemma \ref{n12} describes the following large-time behavior, which will play a central role in the forthcoming proof.

\begin{lem}\label{stab}
If $\varphi\in C(M)$ satisfies $\varphi>v^+_\lambda$, then $T^\lambda_t \varphi$ uniformly converges to $u_\lambda$ as $t\to+\infty$.
\end{lem}

In what follows, without loss of generality, we assume that $i_0 = 1$. We first show that the family $\{u_\lambda\}_{\lambda\in(0,1)}$ is uniformly bounded. Since $u_\lambda \geqslant v_0$, it suffices to establish a uniform upper bound for $u_\lambda$. As a first step, we prove that $u_\lambda$ is uniformly bounded from above on $M_1$, using the fact that $a(x) > 0$ on $M_1$.

\begin{lem}\label{umaxstab}
For $\lambda>0$, $u_\lambda$ is uniformly bounded from above on $M_1$.
\end{lem}
\begin{proof}
We are going to prove that the following maximum
\[\max_{x\in M_1}(u_\lambda(x)-v_0(x))=u_\lambda(x_0)-v_0(x_0)\]
is attained at $x_0\in M_1$, where $u_\lambda(x_0)\leqslant \frac{A}{\min_{x\in M_1}a(x)}$. Otherwise, we assume that $u_\lambda(x_0)>\frac{A}{\min_{x\in M_1}a(x)}$. Taking a $v_0$-calibrated curve $\gamma:\mathbb R\to M$ satisfying $\gamma(0)=x_0$. Since $M_1$ does not intersect $M_i$, $i\in\{2,\dots,k\}$, the image of the curve $\gamma$ is contained in $M_1$. By continuity, for $t_0>0$ small enough, we have $u_\lambda(\gamma(t))>\frac{A}{\min_{x\in M_1}a(x)}\geqslant \frac{A}{a(\gamma(t))}$ for all $t\in(-t_0,0)$. Thus,
\begin{align*}
u_\lambda(x_0)-u_\lambda(\gamma(-t))&\leqslant \int_{-t}^0 \bigg[L(\gamma(s),\dot\gamma(s))+c_0-\lambda a(\gamma(s))u_\lambda(\gamma(s))+A\lambda\bigg]\, ds
\\ &<\int_{-t}^0 [L(\gamma(s),\dot\gamma(s))+c_0]\, ds=v_0(x_0)-v_0(\gamma(-t)),
\end{align*}
which implies
\[u_\lambda(x_0)-v_0(x_0)<u_\lambda(\gamma(-t))-v_0(\gamma(-t)).\]
This gives a contradiction.

Therefore, for all $x\in M_1$, we have
\[u_\lambda(x)-v_0(x)\leqslant u_\lambda(x_0)-v_0(x_0)\leqslant \frac{A}{\min_{x\in M_1}a(x)}-v_0(x_0),\]
which gives a uniform upper bound $2\|v_0\|_\infty+\frac{A}{\min_{x\in M_1}a(x)}$ of $u_\lambda$ on $M_1$.
\end{proof}

Using the fact that $u_\lambda\prec \bar L$, we can show that $u_\lambda$ is bounded from above on the whole space $M$, once it is bounded from above at one point in $M$.

\begin{lem}\label{bM}
The family $\{u_\lambda\}_{\lambda\in (0,1)}$ is bounded from above on the whole $M$.
\end{lem}
\begin{proof}
Picking an arbitrary point $x\in M$. Let $y\in M_1$. We take a geodesic $\alpha:[0,1]\to M$ with constant speed and connecting $y$ and $x$. We denote by $K_1$ the upper bound of $u_\lambda$ on $M_1$. If $u_\lambda(x)\leqslant K_1$, then we have obtained the upper bound of $u_\lambda$. Thus, we only need to consider the case $u_\lambda(x)>K_1$. Then by continuity, there exists $\sigma\in [0,1)$ such that $u_\lambda(\alpha(\sigma))=K_1$ and $u_\lambda(\alpha(s))>K_1$ for all $s\in (\sigma,1]$. For $s\in (\sigma,1]$, we have
\begin{equation*}
\begin{aligned}
  &u_\lambda(\alpha(s))-u_\lambda(\alpha(\sigma))
  \\ &\leqslant \int_\sigma^s \bigg[L(\alpha(\tau),\dot \alpha(\tau))+c_0-\lambda a(\alpha(\tau))u_\lambda(\alpha(\tau))+A\lambda\bigg]\, d\tau
  \\ &=\int_\sigma^s \bigg[L(\alpha(\tau),\dot \alpha(\tau))+c_0-\lambda a(\alpha(\tau))(u_\lambda(\alpha(\tau))-K_1)-\lambda a(\alpha(\tau))K_1+A\lambda\bigg]\, d\tau
  \\&\leqslant \max_{x\in M,|\dot x|\leqslant \textrm{diam}(M)}|L(x,\dot x)+c_0|+\|a\|_\infty K_1+A+\lambda\|a\|_\infty\int_\sigma^s (u_\lambda(\alpha(\tau))-K_1)\, d\tau.
\end{aligned}
\end{equation*}
Using the Gronwall inequality, we get
\[u_\lambda(\alpha(s))-K_1\leqslant \bigg(\max_{x\in M,|\dot x|\leqslant \textrm{diam}(M)}|L(x,\dot x)+c_0|+\|a\|_\infty K_1+A\bigg)e^{\|a\|_\infty},\quad \forall\lambda\in(0,1).\]
Taking $s=1$, we get a uniform upper bound of $\{u_\lambda\}_{\lambda\in (0,1)}$.
\end{proof}

\begin{lem}\label{s3}
The family $\{u_\lambda\}_{\lambda\in (0,1)}$ is uniformly bounded and equi-Lipschitz continuous.
\end{lem}
\begin{proof}
We have already shown that $\{u_\lambda\}_{\lambda\in (0,1)}$ is uniformly bounded. Now we are going to show that $\{u_\lambda\}_{\lambda\in (0,1)}$ is equi-Lipschitz continuous. Picking $x,y\in M$. We take a geodesic $\alpha:[0,d(x,y)]\to M$ with constant speed connecting $x$ and $y$. It follows that
\begin{equation*}
\begin{aligned}
  u_\lambda(y)-u_\lambda(x)&\leqslant \int_0^{d(x,y)}\bigg[L(\alpha(s),\dot{\alpha}(s))+c_0-\lambda a(\alpha(s)) u_\lambda(\alpha(s))+A\lambda\bigg]\, ds
  \\ &\leqslant \bigg(\max_{x\in M,|\dot x|\leqslant 1}|L(x,\dot x)+c_0|+A+\|a\|_\infty\|u_\lambda\|_\infty\bigg) d(x,y).
\end{aligned}
\end{equation*}
Exchanging $x$ and $y$, we can show that $\{u_\lambda\}_{\lambda\in (0,1)}$ is equi-Lipschitz continuous.
\end{proof}

According to the Arzel\'a-Ascoli theorem and Lemma \ref{s3}, any sequence $(u_{\lambda_n})_n$ with $\lambda_n\to 0^+$ admits a subsequence which uniformly converges to a continuous function $u^*$. According to Proposition \ref{stability}, $u^*$ is a solution of \eqref{E0}.
\begin{lem}\label{leqA}
Let $u_*$ be a limit point of $\{u_\lambda\}_{\lambda\in (0,1)}$. Then
\[\int_{TM}a(x)u_*(x)\, d\tilde\mu\leqslant A,\quad \forall \tilde\mu\in\widetilde{\mathfrak M}.\]
\end{lem}
\begin{proof}
By Proposition \ref{Mather}, $\int_{TM}L(x,\dot x)\, d\tilde\mu=-c_0$ for all $\tilde\mu\in\widetilde{\mathfrak M}$. We then conclude that
\begin{align*}
\int_{TM}(L(x,\dot x)-\lambda a(x)u_\lambda(x)+A\lambda)\, d\tilde{\mu}&=-c_0+\int_{TM}(-\lambda a(x)u_\lambda(x)+A\lambda)\, d\tilde{\mu}
\\ &\geqslant \min_{\tilde\nu}\int_{TM}(L(x,\dot x)-\lambda a(x)u_\lambda(x)+A\lambda)\, d\tilde{\nu}=-c_0,
\end{align*}
where the minimum is taken among all closed probability measures on $TM$. The above inequality implies that
\[\int_{TM}a(x)u_\lambda(x)\, d\tilde{\mu}\leqslant A.\]
Letting $\lambda\to 0^+$, we get the conclusion.
\end{proof}

Fixing $z\in M$. Let $\gamma^z_\lambda:(-\infty,0]\to M$ be a $u_\lambda$-calibrated curve satisfying $\gamma^z_\lambda(0)=z$.
\begin{lem}\label{gameq}
The family $\{\gamma^z_\lambda\}_{z\in M,\ \lambda\in (0,1)}$ is equi-Lipschitz continuous.
\end{lem}
\begin{proof}
By Lemma \ref{s3}, there are two constants $K,\bar \kappa>0$ independent of $\lambda$ such that $u_\lambda$ is bounded by $K$ and $u_\lambda$ is $\bar \kappa$-Lipschitz continuous. By the superlinearity of $L$, for each $T>0$, there is $C_T\in\mathbb R$ such that
\[L(x,v)\geqslant T\|v\|_x+C_T.\]
Thus, for $0\geqslant t>s$, we have
\begin{align*}
\bar \kappa d(\gamma^z_\lambda(t),\gamma^z_\lambda(s))&\geqslant u_\lambda(\gamma^z_\lambda(t))-u_\lambda(\gamma^z_\lambda(s))
\\ &=\int_s^t \bigg[L(\gamma^z_\lambda(\tau),\dot{\gamma}^z_\lambda(\tau))+c_0-\lambda a(\gamma^z_\lambda(\tau))u_\lambda(\gamma^z_\lambda(\tau))+A\lambda\bigg]\, d\tau
\\ &\geqslant \int_s^t \bigg[(\bar \kappa+1)\|\dot{\gamma}^z_\lambda(\tau)\|_{\gamma^z_\lambda(\tau)}+C_{\bar \kappa+1}\bigg]\, d\tau-\|a\|_\infty K(t-s)
\\ &\geqslant (\bar \kappa+1)d(\gamma^z_\lambda(t),\gamma^z_\lambda(s))+(C_{\bar \kappa+1}-\|a\|_\infty K)(t-s),
\end{align*}
which implies that
\[d(\gamma^z_\lambda(t),\gamma^z_\lambda(s))\leqslant (\|a\|_\infty K-C_{\bar \kappa+1})(t-s).\]
The proof is now complete.
\end{proof}

Define
\begin{equation}\label{muz}
  \int_{TM} f(x,\dot x)\, d\tilde{\mu}^{z,\lambda}_{t}:=\frac{1}{t}\int_{-t}^0f(\gamma^z_\lambda(s),\dot{\gamma}^z_\lambda(s))\, ds,\quad \forall f\in C_c(TM).
\end{equation}
According to Lemma \ref{gameq}, the family $\{\tilde{\mu}^{z,\lambda}_{t}\}_{\lambda\in(0,1),\, t>0}$ is tight. Hence, for each fixed $\lambda\in(0,1)$, there exists a sequence $t_n\to +\infty$ such that $(\tilde{\mu}^{z,\lambda}_{t_n})_n$ weakly$^*$-converges to a probability measure $\tilde{\mu}_{z,\lambda}$. Moreover, there exists a sequence $\lambda_n\to 0^+$ such that $(\tilde{\mu}_{z,\lambda_n})_n$ weakly$^*$-converges to a probability measure $\tilde{\mu}_z$.

\begin{lem}\label{mu*M}
Let $\tilde\mu_z$ be a weak$^*$-limit defined above. For each $z\in M$, the limit $\tilde{\mu}_z\in\widetilde{\mathfrak M}$.
\end{lem}
\begin{proof}
{\bf The measure $\tilde{\mu}_z$ is closed.} Let $f\in C^1(M)$, we have
\begin{align*}
\int_{TM} df(\dot x)\, d\tilde{\mu}_{z,\lambda}&=\lim_{n\to +\infty}\frac{1}{t_n}\int_{-t_n}^0\frac{df}{ds}(\gamma^z_{\lambda}(s))\, ds
\\ &=\lim_{n\to +\infty}\frac{f(\gamma^z_{\lambda}(0))-f(\gamma^z_{\lambda}(-t_n))}{t_n}
\leqslant \lim_{n\to +\infty}\frac{2\|f\|_\infty}{t_n}=0.
\end{align*}
Therefore, $\tilde{\mu}_{z,\lambda}$ is closed. Then $\tilde{\mu}_z$ is also closed, since it is a limit of closed measures.

{\bf The measure $\tilde{\mu}_z$ is minimizing.} Since $u_\lambda$ is uniformly bounded,
\begin{align*}
&\int_{TM} [L(x,\dot x)+c_0-\lambda a(x)u_{\lambda}(x)+A\lambda]\, d\tilde{\mu}_{z,\lambda}
\\ &=\lim_{n\to +\infty}\frac{1}{t_n}\int_{-t_n}^0\bigg[L(\gamma^z_{\lambda}(s),\dot{\gamma}^z_{\lambda}(s))+c_0-\lambda a(\gamma^z_{\lambda}(s))u_{\lambda}(\gamma^z_{\lambda}(s))+A\lambda\bigg]\, ds
\\ &=\lim_{n\to +\infty}\frac{u_{\lambda}(\gamma^z_{\lambda}(0))-u_{\lambda}(\gamma^z_{\lambda}(-t_n))}{t_n}=0.
\end{align*}
Let $(\tilde{\mu}_{z,\lambda_n})_n$ weakly$^*$-converges to $\tilde{\mu}_z$, we have
\[\int_{TM} L(x,\dot x)\, d\tilde\mu_z=\lim_{n\to+\infty}\int_{TM} [L(x,\dot x)-\lambda_n a(x)u_{\lambda_n}(x)+A\lambda_n]\, d\tilde{\mu}_{z,\lambda_n}=-c_0,\]
which implies that $\tilde{\mu}_z$ is minimizing.
\end{proof}

\begin{lem}\label{mu*>0}
For each $z\in M$ and $\lambda>0$, we have $\int_{TM} a(x)\, d\tilde{\mu}_{z,\lambda}\geqslant 0$.
\end{lem}
\begin{proof}
We argue by contradiction. Assume there is $z\in M$, $\lambda>0$ and $B_{z,\lambda}>0$ such that
\[\int_{TM} a(x)\, d\tilde{\mu}_{z,\lambda}<-B_{z,\lambda}.\]
Then there is a sequence $t_n\to+\infty$, and $N>0$ such that if $n\geqslant N$, we have
\[\frac{1}{t_n}\int_{-t_n}^0a(\gamma^z_{\lambda}(s))\, ds<-B_{z,\lambda}.\]
We take a small constant $\varepsilon>0$ and define $u_\varepsilon:=u_{\lambda}-\varepsilon$. By Lemma \ref{stab}, there is $t_{\lambda}>0$ depending on $\lambda$ such that
\begin{equation}\label{T>del}
  u_{\lambda}(z)-T^{\lambda}_tu_\varepsilon(z)<\varepsilon,\quad \forall t>t_{\lambda}.
\end{equation}
Note that, since the convergence in Lemma \ref{stab} is uniform, $t_\lambda$ is independent of $z$. Define $s_n:=t_n+t_{\lambda}$, and
\[w_\lambda(s):=u_{\lambda}(\gamma^z_\lambda(s))-T^{\lambda}_{s+s_n}u_\varepsilon(\gamma^z_\lambda(s)).\]
Since $\gamma^z_\lambda$ is a $u_\lambda$-calibrated curve, we have
\[\frac{du_{\lambda}}{ds}(\gamma^z_\lambda(s))=L(\gamma^z_\lambda(s),\dot{\gamma}^z_\lambda(s))+c_0-\lambda a(\gamma^z_\lambda(s))u_{\lambda}(\gamma^z_\lambda(s))+A\lambda,\quad a.e.\ s<0.\]
Moreover, by the definition of the solution semigroup, it follows that
\begin{align*}
\frac{dT^{\lambda}_{s+s_n}u_\varepsilon}{ds}(\gamma^z_\lambda(s))
\leqslant &L(\gamma^z_\lambda(s),\dot{\gamma}^z_\lambda(s))+c_0
\\ &-\lambda a(\gamma^z_\lambda(s))T^{\lambda}_{s+s_n}u_\varepsilon(\gamma^z_\lambda(s))+A\lambda,\quad a.e.\ s\in(-s_n,0).
\end{align*}
We conclude that
\[\dot w_\lambda(s)\geqslant -\lambda a(\gamma^z_\lambda(s))w_\lambda(s),\quad a.e.\ s\in(-s_n,0).\]
Notice that
\[w_\lambda(-s_n)=u_{\lambda}(\gamma^z_\lambda(-s_n))-u_\varepsilon(\gamma^z_\lambda(-s_n))=\varepsilon.\]
When $t_n$ is large enough, we have
\begin{align*}
w_{\lambda}(0)&=u_{\lambda}(z)-T^{\lambda}_{s_n}u_\varepsilon(z)
\\ &\geqslant \varepsilon e^{-\int_{-t_n-t_{\lambda}}^{-t_n}\lambda a(\gamma_\lambda(s))\, ds}e^{-\int_{-t_n}^{0}\lambda a(\gamma_\lambda(s))\, ds}>\varepsilon e^{-\lambda \|a\|_\infty t_{\lambda}}e^{\lambda B_{z,\lambda} t_n}>\varepsilon,
\end{align*}
which contradicts \eqref{T>del}.
\end{proof}

\begin{lem}\label{h=w}
Let $w$ be a subsolution of \eqref{E0}. Let $x_1,x_2\in M$. If $d_H(x_1,x_2)=0$, we have $h^\infty(x_1,x_2)=w(x_2)-w(x_1)$.
\end{lem}
\begin{proof}
By Proposition \ref{h>w} we know that for a subsolution $w$ of \eqref{E0},
\[h^\infty(x_1,x_2)\geqslant w(x_2)-w(x_1).\]
Assume that
\[h^\infty(x_1,x_2)>w(x_2)-w(x_1).\]
Adding $h^\infty(x_2,x_1)$ to both sides of the above inequality, we obtain
\[0=h^\infty(x_2,x_1)+h^\infty(x_1,x_2)>h^\infty(x_2,x_1)+w(x_2)-w(x_1),\]
which implies
\[w(x_1)-w(x_2)>h^\infty(x_2,x_1).\]
This leads to a contradiction.
\end{proof}

\begin{cor}\label{h-h=c}
By Proposition \ref{h>w}, $-h^\infty(\cdot,x)$ is a subsolution of \eqref{E0}. If $d_H(x_1,x_2)=0$, then by Lemma \ref{h=w},
\[h^\infty(x_1,x_2)=-h^\infty(x_2,x)-(-h^\infty(x_1,x)),\]
that is, $h^\infty(x_2,x)$ and $h^\infty(x_1,x)$ differ by a constant $h^\infty(x_1,x_2)$.
\end{cor}

To apply the method introduced in \cite{DNYZ}, it is crucial to establish Lemma \ref{e1e2} below.
By Lemmas \ref{mu*M} and \ref{mu*>0}, together with assumption (a), we know that the limiting measure $\mu_z$ satisfies $\mu_z(M_1)>0$. Consequently, when $\lambda$ is small,
any $u_\lambda$-calibrated curve $\gamma^z_\lambda$ must spend an arbitrarily long time in a
small neighborhood of $M_1$.

However, this observation alone is not sufficient to prove Lemma \ref{e1e2}.
Indeed, we also need to show that, for $j\in\{2,\dots,k\}$, the curve $\gamma^z_\lambda$
can spend only a finite amount of time in a small neighborhood of $M_j$ when $\lambda$ is
small, and moreover that this time is bounded by a constant independent of $\lambda$.
To establish this property, we shall use the definition of static classes and adopt a
dynamical viewpoint.

In Lemma \ref{alp} below, we show that Case A in Figure \ref{fig1} cannot occur; namely,
$\gamma^z_\lambda$ cannot remain for an infinite time in a neighborhood of $M_j$
($j\in\{2,\dots,k\}$) for a fixed small $\lambda$.
Nevertheless, another scenario may happen: the curve $\gamma^z_\lambda$ may first enter
a neighborhood of some $M_j$ and then move into a neighborhood of $M_1$, while the time
$t_\lambda$ spent near $M_j$ tends to $+\infty$ as $\lambda\to0^+$.
This situation is illustrated by Case B in Figure \ref{fig1}.

To avoid dealing with this difficulty directly, we first fix $z\in M_1$ and prove the
convergence of $u_\lambda$ on $M_1$. We then extend the convergence to the whole space $M$
by using Proposition \ref{uniqueM} together with Lemma \ref{eM1}.
The advantage of this strategy is that, when $z\in M_1$, Lemma \ref{alp} implies that
$\gamma^z_\lambda$ cannot enter a neighborhood of $M_j$ for any $j\in\{2,\dots,k\}$.
As a result, Lemma \ref{e1e2} follows easily.

\begin{figure}[htbp]
\small \centering
\includegraphics[width=15cm]{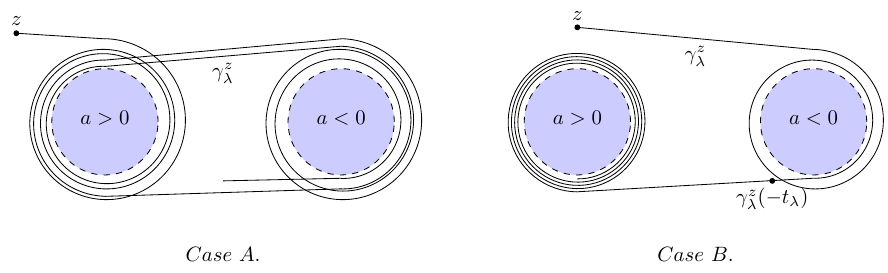}
\caption{Two cases causing difficulties in the proof of Lemma \ref{e1e2}.}\label{fig1}
\end{figure}

\begin{defn}
For $\delta>0$, we denote by $U^\delta$ the open $\delta$-neighborhood of a
subset $U\subset M$, that is,
\[
U^\delta := \{x\in M : \dist(x,U)< \delta\},
\]
where $\dist(x,U):=\inf_{y\in U} d(x,y)$.

Let $U_1$ and $U_2$ be two disjoint open subsets of $M$.
We say that a curve $\gamma:(-\infty,0]\to M$ leaves $U_1$ at time $-t_1$
and enters $U_2$ at time $-t_2$ if
\[
\gamma(-t_1)\in \partial U_1,\qquad \gamma(-t_2)\in \partial U_2,
\]
and
\[
\gamma(t)\notin \overline{U}_1\cup \overline{U}_2
\quad \text{for all}\quad t\in(-t_2,-t_1).
\]

We say that a curve $\gamma$ visits
$M_{i_1}^\delta,M_{i_2}^\delta,\dots,M_{i_p}^\delta$ in order if
$\gamma$ leaves $M_{i_1}^\delta$ at time $-s_1$ and enters
$M_{i_2}^\delta$ at time $-t_2$, then leaves $M_{i_2}^\delta$ at time
$-s_2$ and enters $M_{i_3}^\delta$ at time $-t_3$, and so on.
Finally, $\gamma$ enters $M_{i_p}^\delta$ at time $-t_p$.
Moreover, for each $i\in\{1,\dots,p-1\}$, the curve $\gamma$ does not
intersect the closed $\delta$-neighborhood $\overline{\mathcal M}^\delta$
of $\mathcal M$ on the interval $(-t_{i+1},-s_i)$.
\end{defn}

\begin{lem}
Given $\delta>0$, suppose that there exist $z\in M$ and a sequence
$\lambda_n\to 0^+$ such that the calibrated curve
$\gamma_n:=\gamma^z_{\lambda_n}$ leaves $M_i^\delta$ at time $-t_n^1$
and enters $M_j^\delta$ at time $-t_n^2$, where
$i,j\in\{1,\dots,k\}$ with $i\neq j$.
Assume in addition that $\gamma_n$ does not intersect
$\overline{\mathcal M}^\delta$ on $(-t_n^2,-t_n^1)$.
Then there exists a constant $T_\delta>0$, depending only on $\delta$,
such that
\[t_n^2 - t_n^1 \leqslant T_\delta \qquad \text{for all}\quad n.\]
\end{lem}
\begin{proof}
We argue by contradiction. Assume there are subsequences, which are still denoted by $(t^1_n)_n$ and $(t^2_n)_n$, which satisfy $t^2_n-t^1_n\to +\infty$ as $n\to+\infty$. Define
\[\int_{TM} f(x,\dot x)\, d\tilde{\mu}_{n}:=\frac{1}{t^2_{n}-t^1_{n}}\int_{-t^2_{n}}^{-t^1_{n}}f(\gamma_{n}(s),\dot{\gamma}_{n}(s))\, ds,\quad \forall f\in C_c(TM).\]
According to Lemma \ref{gameq}, $\gamma_{n}$ is equi-Lipschitz continuous. Thus, there is a subsequence of $(\tilde\mu_{n})_n$, still denoted by $(\tilde\mu_{n})_n$, which weakly$^*$-converges to a probability measure $\tilde\mu$. In the following, we are going to show that $\tilde\mu$ is a Mather measure.

\medskip

{\bf The measure $\tilde{\mu}$ is closed.} Let $f\in C^1(M)$, we have
\begin{align*}
\int_{TM} df(\dot x)\, d\tilde{\mu}&=\lim_{n\to +\infty}\frac{1}{t^2_n-t^1_n}\int_{-t^2_n}^{-t^1_n}\frac{df}{ds}(\gamma_n(s))\, ds
\\ &=\lim_{n\to +\infty}\frac{f(\gamma_n(-t^1_n))-f(\gamma_n(-t^2_n))}{t^2_n-t^1_n}
\leqslant \lim_{n\to +\infty}\frac{2\|f\|_\infty}{t^2_n-t^1_n}=0.
\end{align*}

{\bf The measure $\tilde{\mu}$ is minimizing.} Since $u_\lambda$ is uniformly bounded, we have
\begin{align*}
&\int_{TM} [L(x,\dot x)+c_0]\, d\tilde\mu
\\ &=\lim_{n\to+\infty}\int_{TM} [L(x,\dot x)+c_0-\lambda_n a(x)u_{\lambda_n}(x)+A\lambda_n]\, d\tilde{\mu}_n
\\ &=\lim_{n\to+\infty}\frac{1}{t^2_n-t^1_n}\int_{-t^2_n}^{-t^1_n}\bigg[L(\gamma_n(s),\dot{\gamma}_n(s))+c_0-\lambda_n a(\gamma_n(s))u_{\lambda_n}(\gamma_n(s))+A\lambda_n\bigg]\, ds
\\ &=\lim_{n\to+\infty}\frac{u_{\lambda_n}(\gamma_n(-t^1_n))-u_{\lambda_n}(\gamma_n(-t^2_n))}{t^2_n-t^1_n}=0.
\end{align*}
By definition of $\tilde\mu_n$, ${\rm supp}(\mu)$ is not contained in $\mathcal M$, which leads to a contradiction.
\end{proof}

\begin{lem}\label{alp0}
Given $\delta>0$, suppose there exist $z\in M$ and a sequence $\lambda_n\to 0^+$ such that the calibrated curve $\gamma_n := \gamma^z_{\lambda_n}$ leaves $M_i^\delta$ at time $-t_n^1$ and enters $M_j^\delta$ at time $-t_n^2$, where $i,j\in \{1,\dots,k\}$ with $i\neq j$. Moreover, assume that $\gamma_n$ does not pass through $\overline{\mathcal M}^\delta$ on the interval $(-t_n^2, -t_n^1)$. Let $u_*$ be a limit of the sequence $(u_{\lambda_n})_n$. Then there exists a $u_*$-calibrated curve $\alpha_\delta$ leaves $M^{\delta}_i$ at time $0$ and enters $M^{\delta}_{j}$ at time $-t_\delta$, and does not pass through $\overline{\mathcal M}^\delta$ on $(-t_\delta,0)$.
\end{lem}
\begin{proof}
Taking a subsequence of $(\lambda_n)_n$, still denoted by $(\lambda_n)_n$,
we may assume that $(u_{\lambda_n})_n$ converges uniformly to $u_*$. Define $\alpha_{n}(t):=\gamma_{n}(t-t^1_{n})$. Since the family $(\gamma_n)_n$ is equi-Lipschitz continuous, there exists
a subsequence of $(\alpha_n)_n$, still denoted by $(\alpha_n)_n$, which
converges uniformly to a curve $\alpha_\delta$ on the interval
$[-T_\delta,0]$. Notice that $\alpha_{n}(0)=\gamma_{n}(-t^1_{n})$ and $\alpha_{n}(t^1_{n}-t^2_{n})=\gamma_{n}(-t^2_{n})$. By the uniform convergence of $(\alpha_{n})_n$, the limit curve $\alpha_\delta$ leaves $M^{\delta}_i$ at the time $0$ and enters $M^{\delta}_{j}$ at the time $-t_\delta$, where $-t_\delta$ is a limit of $t^1_{n}-t^2_{n}$. Moreover, $\alpha_\delta$ does not pass through $\overline{\mathcal M}^\delta$ on $(-t_\delta,0)$. According to Lemma \ref{TM}, we have
\begin{align*}
&\int_{-t}^0[L(\alpha_\delta(s),\dot\alpha_\delta(s))+c_0]\, ds
\\ &\leqslant \liminf_{n\to+\infty}\int_{-t}^0 [L(\alpha_{n}(s),\dot{\alpha}_{n}(s))+c_0]\, ds
\\ &=\liminf_{n\to+\infty}\int_{-t}^0 \bigg[L(\alpha_{n}(s),\dot{\alpha}_{n}(s))+c_0-\lambda_{n} a(\alpha_{n}(s))u_{\lambda_{n}}(\alpha_{n}(s))+A\lambda_{n}\bigg]\, ds
\\ &=\lim_{n\to +\infty}(u_{\lambda_{n}}(\alpha_{n}(0))-u_{\lambda_{n}}(\alpha_{n}(-t)))
\\ &=u_*(\alpha_\delta(0))-u_*(\alpha_\delta(-t)),\quad \forall t\in[-T_\delta,0],
\end{align*}
which implies that $\alpha_\delta$ is a $u_*$-calibrated curve.
\end{proof}

\begin{lem}\label{tinf}
Assume that there exist a sequence $\delta_n\to 0^+$, a solution $u_*$ of
\eqref{E0}, and a sequence of $u_*$-calibrated curves $\alpha_n$ such that
$\alpha_n$ leaves $M_i^{\delta_n}$ at time $0$ and enters $M_j^{\delta_n}$
at time $-t_n$, where $i,j\in\{1,\dots,k\}$ with $i\neq j$.
Then $t_n\to +\infty$ as $\delta_n\to 0^+$.
\end{lem}
\begin{proof}
We argue by contradiction. Assume that there exists a subsequence $\delta_n \to 0^+$ such that the corresponding times $t_n$ are uniformly bounded from above by some $T>0$. Since $(\alpha_n)_n$ is equi-Lipschitz continuous, up to a subsequence still denoted by $(\alpha_n)_n$, we may assume that $(\alpha_n)_n$ converges uniformly to $\alpha_*$ on $[-T,0]$. By Lemma \ref{TM}, $\alpha_*$ is a $u_*$-calibrated curve. Since dist$(\alpha_n(0),M_i)\to 0$ and dist$(\alpha_n(-t_n),M_j)\to 0$, we conclude that $\alpha_*(0)\in M_i$ and $\alpha_*(-t_*)\in M_j$, where $t_*$ is a limit of $(t_n)_n$. Since $M_i$ does not intersect $M_j$, there is a point $t_0\in (-t_*,0)$ such that $\alpha_*(t_0)\notin \mathcal M$, which leads to a contradiction.
\end{proof}

\begin{lem}\label{xi1}
If there is $z\in M$, sequences $\delta_n\to 0^+$ and $\lambda_n\to 0^+$, such that $\gamma_n:=\gamma^z_{\lambda_n}$ visits $M^{\delta_n}_{i_1},\dots,M^{\delta_n}_{i_p},M^{\delta_n}_{i_1}$ in order, then $M_{i_1},\dots,M_{i_p}$ belong to the same static class.
\end{lem}
\begin{proof}
Up to a subsequence, we may assume that $(\delta_n)_n$ is decreasing. By passing to a further subsequence, still denoted by $(\lambda_{n})_n$, we may also assume that $(u_{\lambda_{n}})_n$ uniformly converges to $u_*$. When $n\geqslant N$, we have $\delta_{n}\leqslant \delta_N$, and hence $\gamma_{n}$ visits $M^{\delta_{N}}_{i_1},\dots,M^{\delta_N}_{i_p},M^{\delta_N}_{i_1}$ in order. By Lemma \ref{alp0}, there exists a $u_*$-calibrated curve $\alpha^m_N$ which leaves $M^{\delta_N}_{i_m}$ at time $0$ and enters $M^{\delta_N}_{i_{m+1}}$ at time $-t^m_N$. Moreover, it does not pass through $\overline{\mathcal M}^{\delta_N}$ on the interval $(-t^m_N,0)$. Here $m=1,\dots,p$, and we set $M_{i_{p+1}}=M_{i_1}$.

Since $M$ is compact, up to a subsequence still denoted by $\delta_N\to 0^+$, we may assume that $\alpha^m_N(0)\to \xi_m\in M_{i_m}$ and $\alpha^m_N(-t^m_N)\to \eta_{m+1}\in M_{i_{m+1}}$. Since $\alpha^m_N$ is a $u_*$-calibrated curve, we have
\[h_{t^m_N}(\alpha^m_N(-t^m_N),\alpha^m_N(0))=\int_{-t^m_N}^0 [L(\alpha^m_N(s),\dot{\alpha}^m_N(s))+c_0]\, ds,\]
By Proposition \ref{htlip}, $(x,y)\mapsto h_t(x,y)$ is $\kappa$-Lipschitz continuous. Hence,
\begin{align*}
&h_{t^m_N}(\eta_{m+1},\xi_m)
\\ &\leqslant h_{t^m_N}(\alpha^m_N(-t^m_N),\alpha^m_N(0))+\kappa d(\alpha^m_N(-t^m_N),\eta_{m+1})+\kappa d(\alpha^m_N(0),\xi_m)
\\ &=\int_{-t^m_N}^0 [L(\alpha^m_N(s),\dot{\alpha}^m_N(s))+c_0]\, ds+\kappa d(\alpha^m_N(-t^m_N),\eta_{m+1})+\kappa d(\alpha^m_N(0),\xi_m).
\end{align*}
By Lemma \ref{tinf}, $t^m_N\to+\infty$ as $\delta_N\to 0^+$. Therefore, by the definition of the Peierls barrier, we obtain
\begin{align*}
&h^\infty(\eta_{m+1},\xi_m)
\\ &\leqslant \liminf_{N\to +\infty}h_{t^m_N}(\eta_{m+1},\xi_m)
\\ &\leqslant \liminf_{N\to +\infty}\int_{-t^m_N}^0 [L(\alpha^m_N(s),\dot{\alpha}^m_N(s))+c_0]\, ds+\kappa d(\alpha^m_N(-t^m_N),\eta_{m+1})+\kappa d(\alpha^m_N(0),\xi_m)
\\ &=\liminf_{N\to+\infty}\int_{-t^m_N}^0 [L(\alpha^m_N(s),\dot{\alpha}^m_N(s))+c_0]\, ds
\\ &=\lim_{N\to+\infty}(u_*(\alpha^m_N(0))-u_*(\alpha^m_N(-t^m_N)))=u_*(\xi_m)-u_*(\eta_{m+1}).
\end{align*}
Using the triangle inequality, we get
\[d_H(\xi_1,\eta_2)=h^\infty(\xi_1,\eta_2)+h^\infty(\eta_2,\xi_1)
\leqslant \sum_{i=1}^{p}\bigg[h^\infty(\eta_{i+1},\xi_i)+h^\infty(\xi_{i+1},\eta_{i+1})\bigg],\]
where $\xi_{p+1}:=\xi_1$ and $\xi_1,\eta_{p+1}\in M_{i_1}$. By Lemma \ref{h=w}, we obtain
\[d_H(\xi_1,\eta_2)=\sum_{i=1}^{p}\bigg[u_*(\xi_i)-u_*(\eta_{i+1})+u_*(\eta_{i+1})-u_*(\xi_{i+1})\bigg]=0.\]
Similarly, one can prove that $d_H(\xi_m,\eta_{m+1})=0$, where $m=1,\dots,p$. The proof is now complete.
\end{proof}

\begin{cor}\label{twice}
Given $z \in M$, there exist constants $\delta_0 > 0$ and $\lambda_0 > 0$ such that, for all $\lambda \in (0,\lambda_0)$, the calibrated curve $\gamma^z_\lambda$ does not visit the neighborhood $M_i^{\delta_0}$ of any static class $M_i$ more than once.
\end{cor}
\begin{proof}
Assume that there exist sequences $\delta_n\to 0^+$ and $\lambda_{n}\to 0^+$ such that each $\gamma_{n}:=\gamma^z_{\lambda_n}$ visits $M^{\delta_n}_i$ twice. Then, for each $n$, there exist indices $i_1,\dots,i_q\in\{1,\dots,k\}$, which are distinct from each other, such that $\gamma_n$ visits $M^{\delta_n}_{i_1},\dots,M^{\delta_n}_{i_q},M^{\delta_n}_{i_1}$ in order. Since $k$ is finite, there are only finitely many possible such sequences $(i_1,\dots,i_q)$. Therefore, up to extracting a subsequence of $(\lambda_n)_n$, we may assume that the sequence $(i_1,\dots,i_q)$ is independent of $n$, and that $\gamma_n$ visits $M^{\delta_n}_{i_1},\dots,M^{\delta_n}_{i_q},M^{\delta_n}_{i_1}$ in order for all $n$. By Lemma \ref{xi1}, $M_{i_1},\dots,M_{i_q}$ belong to the same static class, which yields a contradiction.
\end{proof}

\begin{lem}\label{eM1}
For each $\delta>0$ and each $z\in M$, there exists $\lambda_0>0$ such that, for all
$\lambda\in(0,\lambda_0)$, the calibrated curve $\gamma^z_\lambda$ enters $M^\delta_1$
at some time $-t_{z,\lambda}$.
\end{lem}
\begin{proof}
By Lemma \ref{mu*>0}, since we have assumed the condition (a), the limiting measure $\mu_z$ satisfies $\mu_z(M_1)>0$. Thus, for each $\delta>0$ and for each $z\in M$, there is $\lambda_0>0$ such that for all $\lambda\in(0,\lambda_0)$, the calibrated curve $\gamma^z_\lambda$ must enter $M^\delta_1$ at a certain time $-t_{z,\lambda}$.
\end{proof}

\begin{lem}\label{alp}
Given $z\in M$, there exist $\delta_0>0$ and $\lambda_0>0$ such that,
for all $\lambda\in(0,\lambda_0)$, once the calibrated curve
$\gamma^z_\lambda$ enters $M^{\delta_0}_1$, it never enters
$M^{\delta_0}_j$ for any $j\in\{2,\dots,k\}$.
\end{lem}
\begin{proof}
Suppose by contradiction that there exist sequences $\delta_n \to 0^+$ and $\lambda_n \to 0^+$, and an index $j \in \{2,\dots,k\}$, such that $\gamma_{n}:=\gamma^z_{\lambda_n}$ enters $M^{\delta_n}_j$ after passing through $M^{\delta_n}_1$. Let $\mu_z$ be a weak limit of $(\mu_{z,\lambda_n})_n$. For any point $x$ in the support of $\mu_z$, there exists a sequence $x_n\to x$  with $x_n \in {\rm supp}(\mu_{z,\lambda_n})$ for each $n$. Since $\mu_z(M_1)>0$, it follows that for any $\delta>0$, there exists $n_0$ such that for all $n \geqslant n_0$, the curve $\gamma_n$ visits $M^{\delta}_1$. Fix $\delta_n>0$. Choosing an index $m \geqslant n$ such that $\lambda_m$ is sufficiently small, we conclude that $\gamma_m$ must return to $M^{\delta_n}_1$ after passing through $M^{\delta_n}_j$. Letting $\lambda_m\to 0^+$ and $\delta_n\to 0^+$, we deduce that $\gamma_m$ visits $M^{\delta_n}_1$ twice, which contradicts Corollary \ref{twice}.
\end{proof}

Now fix $z\in M_1$. By Lemma \ref{alp}, there exist constants $\delta_0>0$ and $\lambda_0>0$ such that, for all $\lambda\in(0,\lambda_0)$, the calibrated curve $\gamma^z_{\lambda}$ does not enter $M^{\delta_0}_j$ for any $j\in\{2,\dots,k\}$. Since $a(x)$ is positive on $M_1$, there exists $\epsilon>0$ such that
\begin{equation}\label{VM1}
  a(x)>\epsilon,\quad \forall x\in M_1.
\end{equation}
\begin{lem}\label{e1e2}
Fix $z\in M_1$. There exist constants $T_0>0$ and $\lambda_0>0$ such that, for all $\lambda\in(0,\lambda_0)$,
\[\frac{1}{t}\int_{-t}^0 a(\gamma^z_\lambda(s))\, ds>\epsilon,\quad \forall t\geqslant T_0.\]
\end{lem}
\begin{proof}
We argue by contradiction. Assume there is $\lambda_n\to 0^+$ and $t_n\to +\infty$ such that
\[\frac{1}{t_n}\int_{-t_n}^0 a(\gamma_n(s))\, ds\leqslant \epsilon,\]
where we denote $\gamma_n:=\gamma^{z}_{\lambda_n}$ as before. Here we recall the definition in \eqref{muz}. By Lemma \ref{gameq}, the family $(\tilde{\mu}^{z,\lambda_n}_{t_n})_n$ is tight. Hence, up to a subsequence, there exists a limit measure $\tilde{\mu}_*$ of this family. By the inequality above, we also have
\[\int_M a(x)\, d\tilde\mu_*\leqslant \epsilon,\]
In the following, we will show that $\tilde\mu_*$ is a Mather measure.

\medskip

{\bf The measure $\tilde\mu_*$ is closed.} Let $f\in C^1(M)$, we have
\begin{align*}
\int_{TM} df(\dot x)\, d\tilde{\mu}_*&=\lim_{n\to +\infty}\frac{1}{t_n}\int_{-t_n}^{0}\frac{df}{ds}(\gamma_n(s))\, ds
\\ &=\lim_{n\to +\infty}\frac{f(\gamma_n(0))-f(\gamma_n(-t_n))}{t_n}
\leqslant \lim_{n\to +\infty}\frac{2\|f\|_\infty}{t_n}=0.
\end{align*}

{\bf The measure $\tilde\mu_*$ is minimizing.} Since $u_\lambda$ is uniformly bounded, we have
\begin{align*}
&\int_{TM} [L(x,\dot x)+c_0]\, d\tilde\mu_*
\\ &=\lim_{n\to+\infty}\int_{TM} [L(x,\dot x)+c_0-\lambda_n a(x)u_{\lambda_n}(x)+A\lambda_n]\, d\tilde{\mu}_n
\\ &=\lim_{n\to+\infty}\frac{1}{t_n}\int_{-t_n}^{0}\bigg[L(\gamma_n(s),\dot{\gamma}_n(s))+c_0-\lambda_n a(\gamma_n(s))u_{\lambda_n}(\gamma_n(s))+A\lambda_n\bigg]\, ds
\\ &=\lim_{n\to+\infty}\frac{u_{\lambda_n}(\gamma_n(0))-u_{\lambda_n}(\gamma_n(-t_n))}{t_n}=0.
\end{align*}

\medskip

By Lemma \ref{alp}, there exists $\delta_0>0$ such that, for $\lambda_n$ sufficiently small, $\gamma_n$ does not enter $M^{\delta_0}_j$ for any $j\in\{2,\dots,k\}$. Consequently, the support of $\mu_*$ is contained in $M_1$, which contradicts \eqref{VM1}.
\end{proof}

\begin{lem}
Fix $z\in M_1$. There exist constants $T_0>0$ and $\lambda_0>0$ such that, for all $t\geqslant T_0$ and all $\lambda\in(0,\lambda_0)$ we have
\[e^{-\lambda\int_t^0a(\gamma^z_\lambda(s))\, ds}\to 0,\ \textrm{as}\ t\to-\infty,\]
and
\begin{equation}\label{<e<}
  \frac{e^{-\lambda \|a\|_\infty T_0}}{\lambda\|a\|_\infty}\leqslant \int_{-\infty}^0 e^{-\lambda \int_t^0a(\gamma^z_\lambda(s))\, ds}\, dt \leqslant \frac{1}{\lambda \epsilon}+\frac{e^{\lambda \|a\|_\infty T_0}-1}{\lambda\|a\|_\infty}.
\end{equation}
\end{lem}
\begin{proof}
By Lemma \ref{e1e2}, we obtain
\[\epsilon t\leqslant \int_{-t}^0a(\gamma^z_\lambda(s))\, ds\leqslant \|a\|_\infty t.\]
Consequently,
\[e^{-\lambda\int_t^0a(\gamma^z_\lambda(s))\, ds}\leqslant e^{\lambda\epsilon t}\to 0,\ \textrm{as}\ t\to-\infty.\]
Moreover, 
\begin{align*}
\int_{-T}^0 e^{-\lambda \int_t^0a(\gamma^z_\lambda(s))\, ds}\, dt &=\int_{-T}^{-T_0} e^{-\lambda \int_t^0a(\gamma^z_\lambda(s))\, ds}\, dt+\int_{-T_0}^0 e^{-\lambda \int_t^0a(\gamma^z_\lambda(s))\, ds}\, dt
\\ &\leqslant \int_{-T}^{-T_0} e^{\lambda \epsilon t}\, dt+\int_{-T_0}^0 e^{-\lambda \|a\|_\infty t}\, dt
\\ &=\frac{e^{-\lambda\epsilon T_0}-e^{-\lambda\epsilon T}}{\lambda \epsilon}+\frac{e^{\lambda\|a\|_\infty T_0}-1}{\lambda\|a\|_\infty}\leqslant \frac{1}{\lambda \epsilon}+\frac{e^{\lambda\|a\|_\infty T_0}-1}{\lambda\|a\|_\infty}.
\end{align*}
On the other hand,
\begin{align*}
\int_{-T}^0 e^{-\lambda \int_t^0 a(\gamma^z_\lambda(s))\, ds}\, dt&\geqslant \int_{-T}^{-T_0} e^{-\lambda \int_t^0 a(\gamma^z_\lambda(s))\, ds}\, dt
\\ &\geqslant \int_{-T}^{-T_0} e^{\lambda \|a\|_\infty t}\, dt=\frac{e^{-\lambda \|a\|_\infty T_0}-e^{-\lambda \|a\|_\infty T}}{\lambda\|a\|_\infty}.
\end{align*}
Letting $T \to +\infty$ yields the desired conclusion.
\end{proof}

According to \eqref{<e<}, for $z\in M_1$ and $\lambda$ sufficiently small, the following probability measure $\tilde{\mu}_\lambda^z$ is well-defined
\[\int_{TM}f(x,\dot x)\, d \tilde{\mu}_\lambda^z:=\frac{\int_{-\infty}^0 e^{-\lambda\int_t^0a(\gamma^z_\lambda(s))\, ds}f(\gamma^z_\lambda(t), \dot{\gamma}^z_\lambda(t))\, dt}{\int_{-\infty}^0 e^{-\lambda\int_t^0a(\gamma^z_\lambda(s))\, ds}\, dt}.\]

\begin{lem}\label{suponM1}
Let $\tilde{\mu}^z$ be a limit measure of $\tilde{\mu}_\lambda^z$ as $\lambda \to 0^+$. Then $\tilde{\mu}^z$ is a Mather measure. Moreover, its projected measure $\mu^z$ is supported on $M_1$.
\end{lem}
\begin{proof}
{\bf The measure $\tilde{\mu}^z$ is closed.} For $f\in C^1(M)$, we have
\begin{align*}
\int_{TM}df(v)\, d\tilde{\mu}^z_\lambda&=\frac{\int_{-\infty}^0 \frac{d}{dt}(f(\gamma^z_\lambda(t)))e^{-\lambda \int_t^0 a(\gamma^z_\lambda(s))\, ds}\, dt}{\int_{-\infty}^0 e^{-\lambda \int_t^0 a(\gamma^z_\lambda(s))\, ds}\, dt}
\\ &=\frac{f(x)-\int_{-\infty}^0 f(\gamma^z_\lambda(t))\frac{d}{dt}e^{-\lambda \int_t^0 a(\gamma^z_\lambda(s))\, ds}\, dt}{\int_{-\infty}^0 e^{-\lambda \int_t^0 a(\gamma^z_\lambda(s))\, ds}\, dt},
\end{align*}
where we use \eqref{<e<} to get
\begin{align*}
\bigg|\int_{-\infty}^0 f(\gamma^z_\lambda(t))\frac{d}{dt}e^{-\lambda \int_t^0 a(\gamma^z_\lambda(s))\, ds}\, dt\bigg|&=\lambda\bigg|\int_{-\infty}^0 f(\gamma^z_\lambda(t))a(\gamma^z_\lambda(t))e^{-\lambda \int_t^0 a(\gamma^z_\lambda(s))\, ds}\, dt\bigg|
\\ &\leqslant \lambda\|f\|_\infty\|a\|_\infty \int_{-\infty}^0 e^{-\lambda \int_t^0 a(\gamma^z_\lambda(s))\, ds}\, dt
\\ &\leqslant \bigg(\frac{\|a\|_\infty}{\epsilon}+e^{\lambda \|a\|_\infty T_0}-1\bigg)\|f\|_\infty,
\end{align*}
Thus,
\begin{align*}
\int_{TM}df(v)\, d\tilde{\mu}^z_\lambda&\leqslant \lambda\|a\|_\infty e^{\lambda \|a\|_\infty T_0}\bigg(\frac{\|a\|_\infty}{\epsilon}+e^{\lambda \|a\|_\infty T_0}\bigg)\|f\|_\infty\to 0,
\end{align*}
as $\lambda \to 0^+$.

{\bf The measure $\tilde{\mu}^z$ is minimizing.} Since $\gamma^z_\lambda$ is a $u_\lambda$-calibrated curve, we have
\[\frac{d}{dt}u_\lambda(\gamma^z_\lambda(t))=L(\gamma^z_\lambda(t),\dot{\gamma}^z_\lambda(t))+c_0-\lambda a(\gamma^z_\lambda(t))u_\lambda(\gamma^z_\lambda(t))+A\lambda,\quad a.e.\ t<0.\]
It follows that
\begin{equation}\label{0}
\begin{aligned}
&\int_{TM}L(x,v)\, d\tilde{\mu}^z_\lambda(x,v)=\frac{\int_{-\infty}^0 L(\gamma^z_\lambda(t),\dot{\gamma}^z_\lambda(t))e^{-\lambda \int_t^0 a(\gamma^z_\lambda(s))\, ds}\, dt}{\int_{-\infty}^0 e^{-\lambda \int_t^0 a(\gamma^z_\lambda(s))\, ds}\, dt}
\\ &=\frac{\int_{-\infty}^0 \bigg(\frac{d}{dt}u_\lambda(\gamma^z_\lambda(t))+\lambda a(\gamma^z_\lambda(t))u_\lambda(\gamma^z_\lambda(t))-A\lambda-c_0\bigg)e^{-\lambda \int_t^0 a(\gamma^z_\lambda(s))\, ds}\, dt}{\int_{-\infty}^0 e^{-\lambda \int_t^0 a(\gamma^z_\lambda(s))\, ds}\, dt}
\\ &\leqslant \frac{\int_{-\infty}^0 \frac{d}{dt}(u_\lambda(\gamma^z_\lambda(t)))e^{-\lambda \int_t^0 a(\gamma^z_\lambda(s))\, ds}\, dt}{\int_{-\infty}^0 e^{-\lambda \int_t^0 a(\gamma^z_\lambda(s))\, ds}\, dt}+\lambda (\|a\|_\infty\|u_\lambda\|_\infty+A)-c_0.
\end{aligned}
\end{equation}
Similar to the proof of the closed property, we obtain
\begin{align*}
\frac{\int_{-\infty}^0 \frac{d}{dt}(u_\lambda(\gamma^z_\lambda(t)))e^{-\lambda \int_t^0 V(\gamma^z_\lambda(s))\, ds}\, dt}{\int_{-\infty}^0 e^{-\lambda \int_t^0 V(\gamma^z_\lambda(s))\, ds}\, dt}\leqslant \lambda \|a\|_\infty e^{\lambda \|a\|_\infty T_0}\bigg(\frac{\|a\|_\infty}{\epsilon}+e^{\lambda \|a\|_\infty T_0}\bigg)\|u_\lambda\|_\infty\to 0,
\end{align*}
as $\lambda \to 0^+$. Letting $\lambda\to 0^+$ in \eqref{0}, we conclude that $\tilde{\mu}^z$ is minimizing.

\medskip

By Lemma \ref{alp}, there exists $\delta_0>0$ such that, for $\lambda$ sufficiently small, $\gamma^z_\lambda$ does not enter $M^{\delta_0}_j$ for any $j\in\{2,\dots,k\}$. Therefore, $\tilde{\mu}^z$ is supported on $M_1$.
\end{proof}

\begin{lem}\label{geqA}
Let $w$ be a subsolution of \eqref{E0}. We have
\[u_\lambda(z)\geqslant w(z)-\frac{\int_{TM}a(x)w(x)\, d\tilde{\mu}^z_\lambda-A}{\int_{TM}a(x)\, d\tilde{\mu}^z_\lambda},\quad z\in M_1.\]
\end{lem}
\begin{proof}
According to Proposition \ref{ue}, for each $\varepsilon>0$, there exists $w_\varepsilon\in C^\infty(M)$ satisying $\|w_\varepsilon-w\|_\infty\leqslant \varepsilon$ and
\[H(x,Dw_\varepsilon(x))\leqslant c_0+\varepsilon,\quad \forall x\in M.\]
It follows that
\[L(x,v)+c_0\geqslant L(x,v)+H(x,Dw_\varepsilon(x))-\varepsilon\geqslant \langle v,Dw_\varepsilon(x)\rangle-\varepsilon,\quad \forall (x,v)\in TM.\]
Hence,
\begin{align*}
\frac{d}{dt}u_\lambda(\gamma^z_\lambda(t))&=L(\gamma^z_\lambda(t),\dot{\gamma}^z_\lambda(t))+c_0-\lambda a(\gamma^z_\lambda(t))u_\lambda(\gamma^z_\lambda(t))+A\lambda
\\ &\geqslant \langle \dot{\gamma}^z_\lambda(t),Dw_\varepsilon(\gamma^z_\lambda(t))\rangle-\varepsilon-\lambda a(\gamma^z_\lambda(t))u_\lambda(\gamma^z_\lambda(t))+A\lambda
\\ &=\frac{d}{dt}w_\varepsilon(\gamma^z_\lambda(t))-\varepsilon-\lambda a(\gamma^z_\lambda(t))u_\lambda(\gamma^z_\lambda(t))+A\lambda,\quad a.e.\ t<0.
\end{align*}
Multiplying both sides by $e^{-\lambda\int_t^0 a(\gamma^z_\lambda(s))\, ds}$, we get
\[\frac{d}{dt}\bigg(u_\lambda(\gamma^z_\lambda(t))e^{-\lambda\int_t^0 a(\gamma^z_\lambda(s))\, ds}\bigg)\geqslant \bigg(\frac{d}{dt}w_\varepsilon(\gamma^z_\lambda(t))-\varepsilon+A\lambda\bigg)e^{-\lambda\int_t^0 a(\gamma^z_\lambda(s))\, ds}.\]
Integrating over the inteval $(-T,0]$, where $T\geqslant T_0$, and using an integration by parts, we obtain
\begin{align*}
&u_\lambda(z)-u_\lambda(\gamma^z_\lambda(-T))e^{-\lambda\int_{-T}^0 a(\gamma^z_\lambda(s))\, ds}
\\ &\geqslant w_\varepsilon(z)-w_\varepsilon(\gamma^z_\lambda(-T))e^{-\lambda\int_{-T}^0 a(\gamma^z_\lambda(s))\, ds}
\\ &\quad -\int_{-T}^0w_\varepsilon(\gamma^z_\lambda(t))\frac{d}{dt}\bigg(e^{-\lambda\int_t^0 a(\gamma^z_\lambda(s))\, ds}\bigg)\, dt
+(A\lambda-\varepsilon) \int_{-T}^0 e^{-\lambda\int_t^0 a(\gamma^z_\lambda(s))\, ds}\, dt
\end{align*}
Letting $T\to +\infty$ and $\varepsilon\to 0$ yields
\[u_\lambda(z)\geqslant w(z)-\int_{-\infty}^0w(\gamma^z_\lambda(t))\frac{d}{dt}\bigg(e^{-\lambda\int_t^0a(\gamma^z_\lambda(s))\, ds}\bigg)\, dt+A\lambda \int_{-\infty}^0 e^{-\lambda\int_t^0a(\gamma^z_\lambda(s))\, ds}\, dt.\]
By the definition of $\tilde\mu^z_\lambda$, we have
\[\int_{-\infty}^0w(\gamma^z_\lambda(t))\frac{d}{dt}\bigg(e^{-\lambda\int_t^0a(\gamma^z_\lambda(s))\, ds}\bigg)\, dt=\lambda\int_{-\infty}^0 e^{-\lambda \int_t^0 a(\gamma^z_\lambda(s))\, ds}\, dt\int_{TM}a(x)w(x)\, d\tilde{\mu}^z_\lambda.\]
A direct calculation yields
\begin{align*}
\lambda \int_{-\infty}^0 e^{-\lambda \int_t^0 a(\gamma^z_\lambda(s))\, ds}\, dt&=\frac{\lambda \int_{-\infty}^0 e^{-\lambda \int_t^0 a(\gamma^z_\lambda(s))\, ds}\, dt}{\int_{-\infty}^0\frac{d}{dt}\bigg(e^{-\lambda\int_t^0a(\gamma^z_\lambda(s))\, ds}\bigg)\, dt}
\\ &=\frac{\int_{-\infty}^0 e^{-\lambda \int_t^0 a(\gamma^z_\lambda(s))\, ds}\, dt}{\int_{-\infty}^0 a(\gamma^z_\lambda(t))e^{-\lambda\int_t^0a(\gamma^z_\lambda(s))\, ds}\, dt}=\frac{1}{\int_{TM}a(x)\, d\tilde{\mu}^z_\lambda}.
\end{align*}
The proof is now complete.
\end{proof}

We define $\mathcal S$ to be the set of all subsolutions of \eqref{E0} satisfying
\[\int_M a(x)w(x)\, d\mu\leqslant A\]
for all projected Mather measures $\mu$ supported on $M_1$. We then define
\[u_0(z):=\sup_{w\in\mathcal S}w(z),\quad z\in M.\]
By Lemma \ref{leqA}, the set $\mathcal S$ is non-empty, which implies that $u_0$ has a lower bound. We also need to show that $u_0$ is bounded from above. Let $w \in \mathcal S$ and assume, by contradiction, that for all $z\in M_1$ we have $w(z)\geqslant A/\epsilon$. Since $a(z)>\epsilon$ on $M_1$, we get
\[\int_{M_1} a(x)w(x)\, d\mu>A,\]
which contradicts
\[\int_{M_1} a(x)w(x)\, d\mu=\int_M a(x)w(x)\, d\mu\leqslant A.\]
Thus, there is a point $z\in M_1$ such that $w(z)\leqslant A/\epsilon$. Since all subsolutions of \eqref{E0} are equi-Lipschitz continuous, it follows that $w$ is uniformly bounded from above on $M$. Consequently, $u_0$ is bounded from above. By Proposition \ref{supsub}, the pointwise supremum of a family of subsolutions is a subsolution, $u_0$ is a subsolution of \eqref{E0}.

\begin{lem}\label{cononM1}
The maximal solution $u_\lambda$ uniformly converges to $u_0$ as $\lambda\to 0^+$ on $M_1$.
\end{lem}
\begin{proof}
Let $u_*$ be a limit of the family $\{u_\lambda\}$. By Lemma \ref{leqA}, we have $u_*\leqslant u_0$. By Lemmas \ref{suponM1} and \ref{geqA}, for each $w\in\mathcal S$ and $z\in M_1$, we have $u_*(z)\geqslant w(z)$. Therefore, we have $u_*(z)=u_0(z)$ for all $z\in M_1$.
\end{proof}

\begin{lem}\label{cononM}
The maximal solution $u_\lambda$ uniformly converges to $u_0$ as $\lambda\to 0^+$ on $M$.
\end{lem}
\begin{proof}
We take a converging sequence $(u_{\lambda_n})_n$, whose limit is denoted by $u_*$. We are going to show that $u_*$ coincides with $u_0$ on $\{M_j\}_{j\in\{2,\dots,k\}}$. Once this is established, Proposition \ref{uniqueM} yields $u_*=u_0$.

Let $j\in\{2,\dots,k\}$ and $z\in M_j$. By relabeling the indices if necessary, we may assume that $j=2$. By Lemma \ref{eM1}, there exist sequences $\delta_n\to 0^+$ and $\lambda_n\to 0^+$ such that $\gamma_n:=\gamma^z_{\lambda_n}$ enters $M^{\delta_n}_1$. By Corollary \ref{twice}, up to a subsequence still denoted by $(\lambda_n)_n$, there exist distinct indices $i_0,i_1,\dots,i_q\in\{2,\dots,k\}$ with $i_0=2$ and $0\leqslant q\leqslant k-2$ such that $\gamma_n$ visits $M^{\delta_n}_{i_0},M^{\delta_n}_{i_1},\dots,M^{\delta_n}_{i_q},M^{\delta_n}_{1}$ in order. Moreover, the sequence of indices $(i_1,\dots,i_q)$ is independent of $n$. Assume that $(\delta_n)_n$ is decreasing. Then for any fixed $N$, and all $n>N$, $\gamma_n$ visits $M^{\delta_N}_{2},M^{\delta_N}_{i_1},\dots,M^{\delta_N}_{i_q},M^{\delta_N}_{1}$ in order. We set $i_{q+1}=1$. By Lemma \ref{alp0}, for each $m=0,\dots,q$, there exists a $u_*$-calibrated curve $\alpha^m_N$ which leaves $M^{\delta_N}_{i_m}$ at time $0$ and enters $M^{\delta_N}_{i_{m+1}}$ at time $-t^m_N$. Up to a subsequence, we have $\alpha^m_N(0)\to \xi_m\in M_{i_m}$ and $\alpha^m_N(-t^m_N)\to\eta_{m+1}\in M_{i_{m+1}}$. By Lemma \ref{tinf}, $t^m_N\to+\infty$ as $N\to+\infty$. Therefore,
\begin{align*}
&h^\infty(\eta_{m+1},\xi_m)\leqslant \liminf_{N\to +\infty}h_{t^m_N}(\eta_{m+1},\xi_m)
\\ &\leqslant \liminf_{N\to +\infty}\int_{-t^m_N}^0 [L(\alpha^m_N(s),\dot{\alpha}^m_N(s))+c_0]\, ds+\kappa d(\alpha^m_N(-t^m_n),\eta_{m+1})+\kappa d(\alpha^m_N(0),\xi_m)
\\ &=\liminf_{N\to+\infty}\int_{-t^m_N}^0 [L(\alpha^m_N(s),\dot{\alpha}^m_N(s))+c_0]\, ds
\\ &=\lim_{N\to+\infty}(u_*(\alpha^m_N(0))-u_*(\alpha^m_N(-t^m_N)))=u_*(\xi_m)-u_*(\eta_{m+1}).
\end{align*}
By Lemma \ref{cononM1}, $u_*(\eta_{q+1})=u_0(\eta_{q+1})$. Taking $m=q$ in the above inequality, we obtain
\[u_*(\xi_q)\geqslant u_0(\eta_{q+1})+h^\infty(\eta_{q+1},\xi_q)\geqslant u_0(\eta_{q+1})+u_0(\xi_q)-u_0(\eta_{q+1})=u_0(\xi_q).\]
By Lemma \ref{h=w}, for all $z\in M_{i_q}$, we have
\[u_*(z)-u_*(\xi_q)=u_0(z)-u_0(\xi_q)=h^\infty(\xi_q,z),\]
which implies
\[u_*(z)=u_0(z)-u_0(\xi_q)+u_*(\xi_q)\geqslant u_0(z).\]
Therefore, $u_*\geqslant u_0$ on $M_{i_q}$. By the definition of $u_0$ we obtain that $u_*(z)=u_0(z)$ for all $z\in M_{i_q}$. Repeating this argument, we conclude that $u_*(z)=u_0(z)$ for all $z\in M_2$.
\end{proof}

\begin{lem}\label{u0h}
Let $u_0$ be the uniform limit of $\{u_\lambda\}$ given in Lemma \ref{cononM}. Then
\[u_0(x)=\inf_{\mu}\frac{\int_M a(y)h^\infty(y,x)\, d\mu(y)+A}{\int_M a(y)\, d\mu(y)}.\]
where the infimum is taken among all projected Mather measures $\mu$ supported on $M_1$.

Moreover, for any fixed point $x_1\in M_1$, the limit function $u_0$ can be written in the form
\[h^\infty(x_1,x)+C,\]
where the constant $C$ is given by
\[C=\inf_{\mu}\frac{\int_M a(y)h^\infty(y,x_1)\, d\mu(y)+A}{\int_M a(y)\, d\mu(y)},\]
where the infimum is taken among all projected Mather measures $\mu$ supported on $M_1$.
\end{lem}
\begin{proof}
Define
\[\hat u_0(x):=\inf_{\mu}\frac{\int_M a(y)h^\infty(y,x)\, d\mu+A}{\int_M a(y)\, d\mu(y)}.\]
where the infimum is taken among all projected Mather measures $\mu$ supported on $M_1$. Since $h^\infty(y,x)\geqslant u_0(x)-u_0(y)$ and $a(y)$ is positive for $y\in M_1$, for all projected Mather measures $\mu$ supported on $M_1$, we have
\begin{align*}
\int_M a(y)h^\infty(y,x)\, d\mu(y)&\geqslant u_0(x)\int_M a(y)\, d\mu(y)-\int_M a(y)u_0(y)\, d\mu(y)
\\ &\geqslant u_0(x)\int_M a(y)\, d\mu(y)-A,
\end{align*}
which implies that $u_0\leqslant \hat u_0$.

On the other hand, for $y\in M$, we define
\[U_y(x):=-h^\infty(x,y)+\hat u_0(y).\]
Then $U_y(x)$ is a subsolution of \eqref{E0} and satisfies
\[\frac{\int_M U_y(x)a(x)\, d\nu(x)-A}{\int_M a(x)\, d\nu(x)}\leqslant 0,\]
where $\nu$ is a projected Mather measure supported on $M_1$. Thus, $U_y\in\mathcal S$, which implies that $U_y(x)\leqslant u_0(x)$. Taking $y\in\mathcal M$, we have $h^\infty(y,y)=0$. Then we obtain that $U_y(y)=\hat u_0(y)\leqslant u_0(y)$ for $y\in\mathcal M$. We then conclude that $\hat u_0\leqslant u_0$ on the whole $M$.

Recall that for each $x\in M$, the map $y\mapsto -h^\infty(y,x)$ is a subsolution of \eqref{E0}. Taking $x_1,y\in M_1$, it follows from Lemma \ref{h=w} that
\[-h^\infty(x_1,x)-(-h^\infty(y,x))=h^\infty(y,x_1).\]
Therefore, for a projected Mather measure $\mu$ supported on $M_1$, we have
\begin{align*}
u_0(x)&=\inf_{\mu}\frac{\int_M a(y)h^\infty(y,x)\, d\mu(y)+A}{\int_M a(y)\, d\mu(y)}
\\ &=\inf_{\mu}\frac{\int_M a(y)(h^\infty(y,x_1)+h^\infty(x_1,x))\, d\mu(y)+A}{\int_M a(y)\, d\mu(y)}
\\ &=\inf_{\mu}\frac{\int_M a(y)h^\infty(y,x_1)\, d\mu(y)+A}{\int_M a(y)\, d\mu(y)}+h^\infty(x_1,x).
\end{align*}
The proof is now complete.
\end{proof}

\section*{Acknowledgements}

Panrui Ni is supported by the National Natural Science Foundation of China (Grant No. 12571197). 
Jun Yan is supported by the National Natural Science Foundation of China (Grant Nos. 12171096, 12231010, 12571197). 
Maxime Zavidovique is supported by ANR CoSyDy (ANR-CE40-0014).

\section*{Declarations}

\noindent {\bf Conflict of interest statement:} The authors state that there is no conflict of interest.

\medskip

\noindent {\bf Data availability statement:} Data sharing not applicable to this article as no datasets were generated or analysed during the current study.

\end{document}